\documentclass[11pt,reqno]{amsart}
\usepackage[leqno]{amsmath}
\usepackage{amscd,amssymb,palatino}
\usepackage[english]{babel}
\usepackage{caption}
\usepackage{etex}
\usepackage[leqno]{amsmath}
\usepackage{amssymb}
\usepackage{color}
\usepackage{shadow}
\usepackage{epsfig}
\usepackage{epic}
\usepackage[dvipsnames]{xcolor}
\usepackage{graphics}
\usepackage{graphicx}
\usepackage{psfrag}
\usepackage{url}
\usepackage{hyperref}
\usepackage{comment}

\usepackage{amsthm, amscd, amsmath, amsfonts, amssymb} 
\usepackage{color}
\newtheorem{theo}{Theorem}
\newtheorem{theorem}[theo]{Theorem}
\newtheorem{cor}[theo]{Corollary}
\newtheorem{definition}{Definition}[section]

\newtheorem{remark}{Remark}[section]
\newtheorem{lemma}{Lemma}[section]
\newtheorem{example}{Example}[section]
\newtheorem{prop}{Proposition}

\DeclareMathOperator{\Ad}{Ad}

\newcommand\restr[2]{{
  \left.\kern-\nulldelimiterspace
  #1
  \vphantom{\big|}
  \right|_{#2}
  }}

\newcommand{\R}{\mathbb{R}}

\usepackage{setspace}
\usepackage{graphicx}

\usepackage{tikz-cd}
\usepackage{circuitikz}
\usetikzlibrary{shapes.geometric, arrows}
\newcommand{\midarrowr}{\tikz \draw[-triangle 90] (0,0) -- +(.1,0);}
\newcommand{\midarrowu}{\tikz \draw[-triangle 90] (0,0) -- +(0,.1);}
\newcommand{\midarrowd}{\tikz \draw[-triangle 90] (0,0.1) -- (0,0);}

\newcommand\DrawGenus[7]{
  \pgfmathsetmacro{\xstart}{#1 - (0.985*#4)}
  \pgfmathsetmacro{\ystart}{#2 + (0.2*#3)}
	\draw[color = #6, rotate around={#5:(#1,#2)}, #7] (\xstart, \ystart) arc (190:350:#4  and #3);
	\draw[color = #6, rotate around={#5:(#1,#2)}, #7] (\xstart, \ystart) arc (190:210:#4  and #3) arc (150:30:#4  and #3) arc (330:350:#4  and #3);
}

\newcommand\DrawFilledGenus[8]{
  \pgfmathsetmacro{\xstart}{#1 - (0.985*#4)}
  \pgfmathsetmacro{\ystart}{#2 + (0.2*#3)}
	\draw[color = #6, rotate around={#5:(#1,#2)}, #7] (\xstart, \ystart) arc (190:350:#4  and #3);
	\draw[color = #6, rotate around={#5:(#1,#2)}, #7] (\xstart, \ystart) arc (190:210:#4  and #3) arc (150:30:#4  and #3) arc (330:350:#4  and #3);
	\draw[color = #6, rotate around={#5:(#1,#2)}, #7, fill = #8] (\xstart, \ystart) arc (190:210:#4  and #3) arc (150:30:#4  and #3) arc (-30:-150:#4  and #3);
}

\newcommand\DrawDonut[7]{
  \pgfmathsetmacro{\fctr}{.08}
  \pgfmathsetmacro{\newwidth}{0.5*#4}
  \pgfmathsetmacro{\newheight}{0.5*#3}
  \draw[color = #6, rotate around={#5:(#1,#2)}, #7] (#1, #2) ellipse (#4  and #3);
  \DrawGenus{#1}{#2}{\newheight}{\newwidth}{#5}{#6}{#7}
}

\newcommand\DrawFilledDonutops[8]{
  \pgfmathsetmacro{\fctr}{.08}
  \pgfmathsetmacro{\newwidth}{0.5*#4}
  \pgfmathsetmacro{\newheight}{0.5*#3}
  \draw[color = #6, rotate around={#5:(#1,#2)}, #7, fill = #6, opacity = .6] (#1, #2) ellipse (#4  and #3);
  \DrawFilledGenus{#1}{#2}{\newheight}{\newwidth}{#5}{#6}{#7}{#8}
}

\tikzstyle{mytheorembox} = [draw=vdgreen, fill=blue!20, very thick, rectangle, rounded corners, inner sep=10pt, inner ysep=15pt]
\tikzstyle{mytheoremfancytitle} =[fill=vdgreen, text=white]

\definecolor{vdblue}{rgb}{0,0,.3}
\definecolor{dblue}{rgb}{0,0,.7}
\definecolor{lblue}{rgb}{.3,.3,1}
\definecolor{vlblue}{rgb}{.7,.7,1}
\definecolor{vvlblue}{rgb}{.9,.9,1}

\definecolor{vdred}{rgb}{.3,0,0}
\definecolor{dred}{rgb}{.7,0,0}
\definecolor{lred}{rgb}{1,.3,.3}
\definecolor{vlred}{rgb}{1,.7,.7}

\definecolor{vdgreen}{rgb}{0,.2,0}
\definecolor{dgreen}{rgb}{0,.4,0}
\definecolor{lgreen}{rgb}{.3,1,.3}
\definecolor{vlgreen}{rgb}{.7,1,.7}

\definecolor{lyellow}{rgb}{1,1,.3}

\definecolor{gray1}{rgb}{0.22,0.22,0.22}
\definecolor{gray2}{rgb}{0.28,0.28,0.28}
\definecolor{gray3}{rgb}{0.36,0.36,0.36}
\definecolor{gray4}{rgb}{0.44,0.44,0.44}
\definecolor{gray5}{rgb}{0.52,0.52,0.52}
\definecolor{gray6}{rgb}{0.6,0.6,0.6}
\definecolor{gray7}{rgb}{0.68,0.68,0.68}
\definecolor{gray8}{rgb}{0.76,0.76,0.76}

\definecolor{color1}{rgb}{1,0,0}
\definecolor{color2}{rgb}{0.98,0,0.816}
\definecolor{color3}{rgb}{0.717,0,1}
\definecolor{color4}{rgb}{0,0,1}
\definecolor{color5}{rgb}{0,1,1}
\definecolor{color6}{rgb}{0,1,0}
\definecolor{color8}{rgb}{1,1,0}
\definecolor{color7}{rgb}{1,0.651,0}

\begin{document}
\title{Reduction theory for singular symplectic manifolds and singular forms on moduli spaces}
\author{Anastasia Matveeva}
\address{ Anastasia Matveeva,
Laboratory of Geometry and Dynamical Systems, Department of Mathematics at UPC and Barcelona Graduate School of Mathematics}
 \thanks{ The project that gave rise to these results received the support of a fellowship from ”la Caixa” Foundation (ID 100010434). The fellowship code is LCF/BQ/DI18/11660046. Anastasia Matveeva and Eva Miranda are partially supported  by the the Spanish State Research Agency, through the grant PID2019-103849GB-I00 of AEI /10.13039/501100011033.  }
\author{Eva Miranda}\address{Eva Miranda,
Laboratory of Geometry and Dynamical Systems, Department of Mathematics, EPSEB, Universitat Polit\`{e}cnica de Catalunya-IMTech
in Barcelona and
\\ CRM Centre de Recerca Matem\`{a}tica, Campus de Bellaterra
Edifici C, 08193 Bellaterra, Barcelona
 }\thanks{  Eva Miranda is supported by the Catalan Institution for Research and Advanced Studies via an ICREA Academia Prize 2016 and an ICREA Academia Prize 2021 and by the Spanish State
Research Agency, through the Severo Ochoa and Mar\'{\i}a de Maeztu Program for Centers and Units
of Excellence in R\&D (project CEX2020-001084-M) } \thanks{
\textbf{Dedicatory:}  \emph{ Marc Herault was about to start his undergraduate thesis with Eva Miranda in cosupervision with Sergei Gukov and Angus Gruen at the California Institute of Technology.  Marc was a charismatic member of the Laboratory of Geometry and Dynamical Systems at UPC. His passion to understand the interaction of mathematics and physics remains an inspiration for us.  } }
 \email{ eva.miranda@upc.edu}

\begin{abstract} The investigation of symmetries of $b$-symplectic manifolds and folded-symplectic manifolds is well-understood when the group under consideration is a torus (see, for instance, \cite{guillemin2015toric,gualtierietal,GMWconvexity} for $b$-symplectic manifolds and \cite{anarita, CM} for folded symplectic manifolds). However, reduction theory has not been set in this realm in full generality. This is fundamental, among other reasons, to advance in the \emph{\lq\lq quantization commutes with reduction"} programme for these manifolds initiated in \cite{GMWgeomq, GMWgeomqbm}. In this article, we fill in this gap and investigate the Marsden-Weinstein reduction theory under general symmetries for general $b^m$-symplectic manifolds and other singular symplectic manifolds, including certain folded symplectic manifolds. In this new framework, the set of admissible Hamiltonian functions is larger than the category of smooth functions as it takes the singularities of the differential forms into account. The quasi-Hamiltonian set-up is also considered and brand-new constructions of (singular) quasi-Hamiltonian spaces are obtained via  a reduction procedure and the fusion product.
\end{abstract}

\dedicatory{Dedicated to the memory of Marc Herault}

\maketitle

\section{Introduction}
    Reduction theory is crucial in  geometry and physics. It reconciles the abstract concept of symmetry of a system with the practical implementation of changes of variables to simplify the system. The intuition that the number of degrees of freedom reduces under the existence of a group symmetry can be encoded as a reduction theorem. The first ones to observe this were probably Emmy Noether \cite{EN} and Sofia Kovalevskaya  who applied the idea of symmetry to actual mechanical systems (see, for instance, \cite{SK}).
    
    The idea can be taken to different levels of sophistication. When reduction is applied to symplectic geometry, an interesting phenomenon occurs: For a group of dimension $k$, the reduction can be doubled and the system can be simplified by  $2k$ degrees of freedom. This fact is known in the literature as Marsden-Weinstein reduction \cite{MW}. In symplectic geometry, the existence of symmetries is very special. Locally any symplectic manifold is a cotangent bundle, and the existence of symmetries on the base manifold $M$ lifts certain actions (Hamiltonian) to the cotangent bundle. In this line of thought, the idea of symplectic reduction reduces by $2 \dim G$ the number of degrees of freedom and produces an actual symplectic manifold of dimension $2n-2 \dim G$ when the action is free. For non-free actions, the structure of the reduced space is that of a stratified manifold (see \cite{SL}), and symplectic orbifolds are obtained for locally free actions (see \cite{GGK}).
   
   The celebrated Marsden-Weinstein theorem \cite{MW} endows the reduced manifold determined by a fixed-energy level and its symmetries with a symplectic structure. Marsden-Weinstein quotients are closely related to several moduli spaces in geometry and, more concretely, to Geometric Invariant Theory. Frances Kirwan   \cite{FK} related classical Geometric Invariant Theory to symplectic quotients. Symplectic quotients are naturally connected to certain moduli spaces. Michael Atiyah and Raoul Bott unveiled the symplectic structure on the space of flat connections in their celebrated article  \cite{AB}. This was just the commencement of a brave new world \cite{NH, BGPH, BGP} building bridges between the geometry and physics community.

   In this article, we extend the concept of Marsden-Weinstein symplectic reduction to include symplectic manifolds with singular structures and extend the admissible Hamiltonian functions beyond smooth functions. We do this for a class of Poisson manifolds that have been recently closely examined: including $b$-symplectic or log-symplectic (and $b^m$-symplectic) manifolds and \textit{certain folded symplectic manifolds}. 
   Several authors considered group actions on these manifolds (see, for instance, \cite{guillemin2015toric,gualtierietal, GMWconvexity, KMS, KM, KM2, MP, BKM, BKM2, GMPS2}). In \cite{guillemin2015toric, GMWconvexity, CGW}, Delzant type polytopes were investigated for toric actions on manifolds endowed with symplectic structures with singularities (of  $b$ or folded type). Toric symmetries have also been used in the study of formal geometric quantization \cite{GMWgeomq, GMWgeomqbm} where the set of Hamiltonian functions extends to $b^m$-functions. However, a typical picture where the reduced manifolds are analyzed is missing in the literature even for the case of $b$-symplectic manifolds. In this article, we fill in this gap in the literature and extend the reduction to both more general singularities and the quasi-Hamiltonian set-up.

   Our motivating example is a moduli space of flat connections on a symplectic surface. It is possible to associate a geometrical template of identified polygons to such a problem (this is classical; see, for instance, the clear exposition in \cite{DM}). We start with a symplectic template and use the desingularization technique of \cite{GMW} to obtain a singular toy model  (of $b^m$-type for even $m$) by an \emph{ad-hoc} construction from a symplectic template. This moduli space which is symplectic can then be also seen as a reduction obtained from a singular model. This toy example inspires us to extend the identification between moduli space and symplectic reduction to the singular realm and formally define the symplectic reduction for arbitrary Lie groups for $b^m$-symplectic manifold. Other motivating examples come from Yang-Mills fields theories on manifolds with boundary (see \cite{MMN}).

    In order to define the Marsden-Weinstein reduction in the singular realm,  we first need to refine the slice theorem for group actions to consider these singularities in the underlying geometrical structure. As the slice theorem gives a normal form for the geometrical structure, it yields a proper structure and group action on the set of orbits induced on the pre-image of a regular point by the moment map. In particular, this defines a reduced space which is symplectic whenever the highest modular weight of the transverse $S^1$-action is non-vanishing. So, in this case, the reduction procedure eliminates the singularity from the original symplectic structure.

    The philosophical approach to the reduction theory in that article is that of simplifying not only the symmetries of the system but also the singularities of the symplectic structure, as we prove when the highest modular weight is non-vanishing.
   
    Other approaches to the removal, blow-up or desingularization of singularities in this theory have been developed by Guillemin-Miranda-Weitsman in \cite{GMW}.
    This desingularization technique in \cite{GMW} will be a close ally in our endeavour as it puts the reduction and the slice theorem for these different singular symplectic manifolds on equal footing. In particular, it allows us to extend the notion of \emph{reduction by stages} to the new category of singular symplectic manifolds and more general Hamiltonian functions.  
    
The idea of reduction also prevails outside the symplectic realm. Quasi-Hamiltonian spaces (confer \cite{AMM, boalch, AMW, HJS}) provide a natural generalization of Hamiltonian spaces and understanding their properties can be revealing in terms of representation theory. As shown in \cite{AMM}, The category of
$G$-quasi-Hamiltonian spaces is equivalent to a subcategory of the category of
infinite-dimensional symplectic manifolds with Hamiltonian actions  of the loop group of $G$. Thus exploring this extension allows us gain understanding of infinite-dimensional analogues as in \cite{tudor}.
    By the same token, we consider quasi-Hamiltonian actions and reduction as a natural completion of the picture, specially guided by our motivating example. In doing so, we also extend the reduction scheme \lq\lq by stages" to the singular quasi-Hamiltonian realm. We obtain new examples of quasi-Hamiltonian spaces by combining classical quasi-Hamiltonian constructions with techniques native to $b^m$-Hamiltonian spaces using the fusion product. These structures can be generalized further to $E$-quasi-Hamiltonian spaces.

    \vspace{2mm}
   
   \textbf{Organization of this paper:  } In section 2, we give the main definitions and preliminary results needed in the article. Section 3 gives a motivating example of a singular structure on a moduli space, taking the Atiyah-Bott space of flat connections on a surface as a starting point. In section 4, we first recall several schemes where slices have been proved to be helpful, and we prove a $b^m$-symplectic slice theorem, which will be a building block for the reduction theory.   In this section, we also investigate the compatibility between different slice theorems and the desingularization procedures explained in \cite{GMW}. In section 5, we  prove the reduction theorem which is the main theorem  in this article and which  can be succinctly  stated as follows:
  \vspace{1mm}
  \begin{center}
    \fbox{\begin{minipage}{16cm}\emph{The $b^m$-Hamiltonian reduction of a \textbf{$b^m$-symplectic manifold} with nonzero highest modular weight is a \textbf{symplectic manifold}. }\end{minipage}}
   \end{center}
  
   \vspace{1mm}
  The reduction removes the singularity from the symplectic structure. So, as a motto \emph{the reduction entails a desingularization}. In section 6, we also prove that this reduction procedure commutes with the desingularization procedure in \cite{GMW} and observe that reduction can be done by stages. In the last section, we discuss several generalizations of these ideas to more general notions of moment map on $b^m$-manifolds which are not necessarily symplectic and focus on the quasi-Hamiltonian reduction in the singular framework. The mnemonics and removal of the singularity works in the singular quasi-Hamiltonian case as the $b^m$-Hamiltonian case. Reduction with the appropriate group gets rid of the singularity of the form. Our constructions via the fusion product provide brand-new examples of non-trivial singular quasi-Hamiltonian spaces. 
  
 \vspace{1mm}

\textbf{Acknowledgements:}  We are grateful to Anton Alekseev for inspiring conversations about quasi-Hamiltonian spaces. We thank Tobias Diez and Tudor Ratiu for several discussions. Thanks are due to Juan Carlos Marrero, and Edith Padrón for calling our attention to the articles \cite{MPR} and \cite{GZ}.

\section{Preliminaries}
    
   The letter "$b$" on $b$-symplectic theory is reminiscent of  the  work on calculus on manifolds with \emph{boundary}  by Melrose \cite{Mel}. However,  in the theory of $b$-symplectic manifolds the notion is extended to consider a hypersurface which, a posteriori, will turn out to be  the critical hypersurface $Z$ of the generalized symplectic structure.
   
   $b$-Manifolds were first introduced by Melrose in his book \cite{Mel} to give proof of the Atiyah-Patodi-Singer theorem using the same conceptual proof as in the Atiyah-Singer theorem for manifolds with boundary. In \cite{guillemin2014symplectic} this framework was extended, and Poisson structures were associated to $b$-forms of degree $2$ as bivector fields that drop rank along a critical hypersurface. These vector fields can be seen as dual to a two-form with singularity along the hypersurface. The present paper concentrates on a generalization of such forms called $b^m$-symplectic, where a more general transversality condition is imposed. $b$-Symplectic forms, defined and extensively studied in the works \cite{Nest, Mel, guillemin2011codimension, guillemin2014symplectic, GL}, can be seen as a particular case of $b^m$-symplectic for $m=1$, when  a defining function $f$  has been chosen for the critical hypersurface $Z$ such that $Z = \{x \in M | f(x) = 0 \}$. We briefly remind the main definitions and concepts of $b^m$-symplectic geometry to use the mentioned framework.

    This section can be used as a glossary with definitions of singular symplectic structures and corresponding normal forms. The reader familiar with this theory can  skip it.
    
    \subsection{$b^m$-Symplectic manifolds}
    
    Analogously to symplectic geometry, $b^m$-symplectic manifolds can be seen as a particular class of Poisson manifolds.
    
    \begin{definition} \label{def:bMan}
        A $b$-\textbf{manifold} $(M, Z)$ is an oriented manifold $M$ together with an oriented hypersurface $Z$. A $b$-map is a map
        $$
                f : (M_1, Z_1) \to (M_2, Z_2)
        $$
        so that $f$ is transverse to $Z_2$ and $f^{-1}(Z_2) = Z_1$.
    \end{definition}

    \begin{definition}
        A \textbf{$b^m$-vector field} is a vector field $v$ on $M$, such that it is tangent to order $m$ at $Z$. A \textbf{$b^m$-form} is a differential form dual to a $b^m$-vector field.
    \end{definition}
    
    The sets of $b^m$-vector fields and $b^m$-forms respectively are  locally generated by $\{x_1^m \frac{\partial}{\partial x_1}, \frac{\partial}{\partial x_2}, \ldots, \frac{\partial}{\partial x_n}\}$ and $\{\frac{dx_1}{x_1^m}, \frac{dx_2}{x_2}, \ldots, \frac{dx_n}{x_n}\}$, respectively. Due to the Serre-Swan theorem \cite{swan1962vector}, given a $b$-manifold $(M, Z)$, there exists a unique vector bundle $^{b^m} TM$ all whose smooth sections are $b^m$-vector fields. Such a bundle is called a \textbf{$b^m$-tangent bundle}. Analogously, a \textbf{$b^m$-cotangent bundle} can be defined either as dual to the tangent one:
    $$
    ^{b^m} T^*M = (^{b^m} TM)^*
    $$
    or as a bundle, all smooth sections of which are $b^m$-forms.
    
    This allows us to introduce $^{b^m} \Omega^k(M)$ as $\bigwedge ^k(^{b^m} T^* M)$ and the associated $b^m$-cohomology $^{b^m} H^*(M)$. The following theorem relates $b^m$-cohomology to de Rham cohomology \cite{scott}.
    
    \begin{theorem}[The $b^m$-Mazzeo-Melrose]
        $^{b^m} H^*(M) \cong H^*(M) \oplus \left ( H^{*-1} (Z) \right )^m$.
    \end{theorem}    
    
    Among $b^m$-forms, we will focus on forms of degree two that mimic symplectic forms in the de Rham complex.
    \begin{definition}
        Let $(M^{2n}, Z)$ be a $b$-manifold, where $Z$ is the critical hypersurface as in \ref{def:bMan}. Let $\omega \in ^{b^m} \Omega^2(M)$ be a closed $b^m$-form. We say that $\omega$ is \textbf{$b^m$-symplectic} if $\omega_p$ is of maximal rank as an element of $\Lambda^2 \left ( ^{b^m} T_p^* M \right )$ for all $p \in M$. We call the triple $(M, Z, \omega)$ a \textbf{$b^m$-symplectic manifold}.
    \end{definition}
    
    It is possible to describe these forms more precisely in a neighbourhood $U$ of the critical set $Z$. Inside $U=Z\times (-\epsilon, \epsilon)$,   $\omega$ may be written as

    \begin{equation}\label{eqn:newlaurent}
        {\omega = \sum_{j = 1}^{m}\frac{df}{f^j} \wedge \pi^*(\alpha_{j}) +  \beta }
    \end{equation}
    \noindent where the $\alpha_j$ are closed one forms on $Z$, $\beta$  is a  closed 2-form on $U$, $\pi:U\longrightarrow Z$ is the projection and $f$ is the defining function for the critical set $Z$. Non-degeneracy of the form $\omega$
    implies that $\beta\vert_{Z}$ is of maximal rank and $\alpha_m$ is nowhere vanishing. $\alpha_m$ defines the symplectic foliation of the Poisson structure associated with $\omega$, and $\beta$ gives the symplectic form on the leaves of this foliation.

    The following theorem pictures the non-existence of local invariants for $b^m$-symplectic forms other than  dimension in a neighbourhood of a point on the critical hypersurface.
    \begin{theorem}[$b^m$-Darboux theorem]
        Let $\omega$ be a $b^m$-symplectic form on $(M^{2n}, Z)$. Let $p \in Z$. Then we can find a local coordinate chart $(x_1, y_1, \ldots , x_n, y_n)$ centered at $p$ such that hypersurface $Z$ is locally defined by $y_1 = 0$ and
        $$
            \omega = dx_1 \wedge \frac{dy_1}{y_1^m} + \sum_{i = 2}^n dx_i \wedge dy_i.
        $$
    \end{theorem}
    
    The proof of this local normal form relies on the path method.
    
    The $b^m$-analogue of the Moser theorem for symplectic manifolds is convenient for analyzing other invariants (local, semilocal, global) and is proved in  \cite{guillemin2014symplectic}.
   The equivariance of the path method yields the following generalization of the Moser path method (see ~\cite{MP} and \cite{guillemin2015toric}) under the additional structure of a group action.
    
    \begin{theorem}[\textbf{Equivariant $b^m$-Moser Theorem}]\label{theorem:bmmoser}
        Let $\omega_0$ and $\omega_1$ be two $b^m$-symplectic forms on $(M,Z)$ defining the same $b^m$-cohomology class $[\omega_0] = [\omega_1]$ for closed $(M^{2n}, Z)$. Assume there exists a path $\omega_t$ of $b^m$-symplectic forms connecting $\omega_0$ and $\omega_1$  then there exist a $b^m$-symplectomorphism
        $$
            \varphi: \left ( M^{2n}, Z\right ) \to \left ( M^{2n}, Z\right )
        $$
        such that $\varphi^*(\omega_1) = \omega_0$.

         If $(M, Z)$ admits an  action of a compact Lie group $G$ preserving the path $\omega_t$, then $\varphi$ can be chosen equivariant with respect to the $G$-action.
    \end{theorem}

     $b^m$-Symplectic manifolds are dual to $b^m$-Poisson, which allows us to describe these objects in two different languages using either bi-vector fields or differential forms. It is then possible to introduce invariants native to Poisson geometry, such as the \emph{modular vector field}.  
     
    Even though the only local invariant of a $b^m$-symplectic manifold is the dimension, it turns out that the geometry induced by the Poisson structure on the critical set $Z$ yields new semilocal invariants. The structure induced on $Z$ is indeed \emph{cosymplectic}.
  
    Using the flow of the modular vector field, we can define a symplectic mapping torus structure of $Z$, as proved in \cite{guillemin2011codimension}). This mapping group structure is also present on the critical set of a $b^m$-symplectic manifold.

    \begin{definition}\label{mappingtorusstructure}
        Let $(M, Z)$ be a $b^m$-symplectic manifold and suppose that $Z$ is compact and connected and that its symplectic foliation has a compact leaf $\mathcal{L}$. Then the critical set $Z$ is a mapping torus which can be explicitly described as follows: taking any modular vector field $v_{mod}$, there exists a  number $c>0$ such that
        $$
            Z \cong \frac{[0, c]\times \mathcal{L}}{(0,x) \sim(c,\phi(x))}
        $$

        where the time $t$-flow of $v_{mod}$ corresponds to the translation by $t$ in the first coordinate. In particular, $\phi$ is the time $c$-flow of $v_{mod}$. 
        
        The number $c>0$ above is called the \textbf{modular period} of $Z$ and does not depend on the choice of the modular vector field $v_{mod}$.
    \end{definition}
    
\subsection{Folded symplectic manifolds}
    A symplectic form $\omega$ on a manifold $M$ induces a natural volume form on the manifold $\omega^n$, sometimes called the Liouville volume. The next level of sophistication is to consider forms $\omega$ such that $\omega^n$ might vanish at some points but with \emph{good transversality properties}. This is precisely the notion of \emph{folded symplectic structures}.
    \begin{definition}
        Let $(M^{2n}, \omega)$ be a manifold with $\omega$ a closed $2$-form such that the map
        $$
            p \in M \mapsto \left ( \omega(p)\right )^n \in \Lambda^{2n} \left ( T^*M \right )
        $$
        is transverse to the zero section, then $Z = \{p \in M | \left (\omega(p) \right )^n = 0 \}$ is a hypersurface and we say that $\omega$ defines a \textbf{folded symplectic structure} on $(M, Z)$ if additionally its restriction to $Z$ has maximal rank. We call the hypersurface $Z$ \textbf{folding hypersurface} and the pair $(M,Z)$ is a \textbf{folded symplectic manifold}.
    \end{definition}
    
    For simplicity, further, we use the normal form for the folded symplectic structures first described by Martinet in \cite{mart}.
    
    \begin{theorem}[Folded Darboux theorem]
        Let $\omega$ be a folded symplectic form on $(M^{2n}, Z)$ and $p \in Z$. Then we can find a local coordinate chart $(x_1, y_1, \ldots, x_n, y_n)$ centered at $p$ such that the hypersurface $Z$ is locally defined by $y_1 = 0$ and
        $$
            \omega = y_1 dx_1 \wedge dy_1 + \sum_{i = 2}^{n} dx_i \wedge dy_i.
        $$
    \end{theorem}
    
\subsection{Relation between $b^m$-symplectic, symplectic and folded symplectic manifolds} 
    
    \phantom{a}

    To relate singular symplectic manifolds to either symplectic or folded symplectic ones, we recall the desingularization theorem first formulated in \cite{GMW}. Observe that the behaviour of desingularization depends on the degree of the singularity.
    
    \begin{theorem} \label{th:deblog}
        { Let $\omega$ be a} $b^m$-symplectic structure  on a compact manifold $M$  and let $Z$ be its critical hypersurface.
        \begin{itemize}
            \item If $m=2k$ is \textbf{even}, there exists  a family of \textbf{symplectic} forms ${\omega_{\epsilon}}$ which coincide with  the $b^{m}$-symplectic form
            $\omega$ outside an $\epsilon$-neighborhood of $Z$ and for which  the family of bivector fields $(\omega_{\epsilon})^{-1}$ converges in
            the $C^{2k-1}$-topology to the Poisson structure $\omega^{-1}$ as $\epsilon\to 0$ .
            
            \item If $m$ is \textbf{odd}, there exists  a family of \textbf{folded symplectic} forms ${\omega_{\epsilon}}$ which coincide with  the $b^{m}$-symplectic form
            $\omega$ outside an $\epsilon$-neighborhood of $Z$.
        \end{itemize}
    \end{theorem}
    
    An immediate consequence of this theorem is that a $b^{2k}$-symplectic manifold also admits a symplectic structure.

    Following \cite{GMW}, we will look at the desingularization process in more detail and explicitly write the desingularizing function for even degree as we use it to construct a vital example later on in this paper.
    
    Recall that one can write a Laurent series of a closed $b^m$-form in a tubular neighbourhood $U$ of $Z$:
    \begin{equation}
    \label{eq:Laurent}
        \omega = \frac{dx_1}{x_1^m} \wedge \left ( \sum \limits_{i = 0}^{m - 1} \pi^* (\alpha_i) x^i \right ) + \beta,
    \end{equation}
    where $\pi : U \to Z$ is the projection of the tubular neighborhood onto $Z$, $\alpha_i$ is a closed smooth de Rham form on $Z$, and $\beta$ is a de Rham form on $M$.
    
    Due to formula \ref{eq:Laurent}, the $b^{2k}$-form can be written as
    \begin{equation}
        \label{eq:b2kDecomp}
        \omega = \frac{d x_1}{x_1^{2k}} \wedge \sum \limits_{i = 0}^{2k - 1} \left ( x^i \alpha_i \right ) + \beta
    \end{equation}
    on a tubular $\varepsilon$-neighbourhood of a given connected component of $Z$.
    
    More details, including the desingularization function (together with its explicit form), will be provided further in Section \ref{sec:DesingSlice}.

\subsection{$b^m$-Symplectic group actions and moment maps}\label{sec:zigzag}

   In order to describe the group actions on singular symplectic manifolds and the corresponding moment maps, we recall the results of two papers: \cite{BKM} and \cite{guillemin2015toric} (check the preprint version for completeness {\em arXiv:1309.1897v1}).

    As a Poisson manifold, $b^m$-symplectic manifolds admit an induced symplectic foliation. The connected components of $M \setminus Z$ are open symplectic leaves of dimension $2n$, and the critical hypersurface $Z$ admits a co-rank $1$ Poisson (cosymplectic) structure.
    
    The first result in \cite{BKM} characterizes groups which acts on a $b$-symplectic manifolds in a non-trivial manner. These are called transverse in \cite{BKM} when the group action acts transversally to the fibers of the mapping torus in $Z$. Such a notion can be extended to the $b^m$-setting.
    
    \begin{theorem}[Braddell, Kiesenhofer, Miranda] \label{th:bmAction}
        Let $G$ be a compact Lie group acting on a compact $b$-symplectic manifold in a transverse way. Then $G$ is either of the form $S^1 \times H$ or $S^1 \times H \phantom{a} mod \phantom{a} \Gamma$, where $\Gamma = \mathbb Z_l \times \mathbb Z_k$ and $\mathbb Z_k$ is a non-trivial cyclic subgroup of $H$.
    \end{theorem}

    Among the class of Poisson action, the one of $b^m$-Hamiltonian actions plays a central role,
    
    \begin{definition} \label{def:bmHamMM}
        The action of $G$ on a $b^m$-symplectic manifold $(M, Z, \omega)$ is called $b^m$-Hamiltonian if there exists a moment map $\mu \in ^{b^m}\mathcal{C}^\infty(M) \otimes \mathfrak{g}^*$ with
        $$
            \iota(\upsilon_\xi) \omega = \left < d \mu, \xi\right >
        $$
        where $\upsilon_\xi$ is the fundamental vector field generated by $\xi$ and the set of $b^m$-functions is $^{b^m} \mathcal{C}^\infty(M) = \left ( \bigoplus \limits_{i = 1}^{m - 1} t^{-i} \mathcal{C}^{\infty}(t) \right ) \oplus ^{b}\mathcal{C}^\infty(M)$ and $^{b}\mathcal{C}^\infty (M) = \{a \log |t| + g, g \in \mathcal{C}^\infty (M)\}$.
    \end{definition}
        
    In other words, the action is $b^m$-Hamiltonian if it preserves the $b^m$-symplectic form and $\iota(\upsilon_\xi) \omega$ is exact.

    
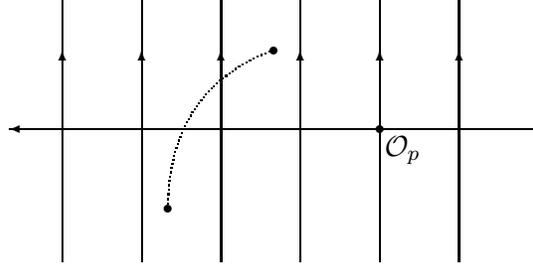
\begin{picture}(300, 100)
        \put(300, 50){\vector(-1, 0){200}}
        \put(120, 0){\line(0, 1){100}}
        \put(120, 50){\vector(0, 1){30}}
        \put(150, 0){\line(0, 1){100}}
        \put(150, 50){\vector(0, 1){30}}
        \put(180, 0){\line(0, 1){100}}
        \put(180, 50){\vector(0, 1){30}}
        \put(210, 0){\line(0, 1){100}}
        \put(210, 50){\vector(0, 1){30}}
        \put(240, 0){\line(0, 1){100}}
        \put(240, 50){\vector(0, 1){30}}
        \put(270, 0){\line(0, 1){100}}
        \put(270, 50){\vector(0, 1){30}}
        \put(240, 50){\circle*{3}}
        \put(242, 40){$\mathcal O_p$}
        \put(160, 20){\circle*{3}}
        \put(200, 80){\circle*{3}}
        \qbezier[50](160, 20)(160, 65)(200, 80)
    \end{picture}{}
    
    For simplicity, let us consider the case of surfaces. Outside the critical hypersurface, locally, the image of the $b$-moment map is just $\mathbb R$, and on the hypersurface, it blows up. 
    
     To explore the difference between Hamiltonian and $b$-Hamiltonian actions and the corresponding moment maps, let us consider the following example:
    \begin{itemize} \label{ex:SpT}
         \item[1.] A Hamiltonian $S^1$-action on the sphere by rotation around the vertical axis with $\R$ being the image of the moment map. (Fig. \ref{pic:s2})
        
        \item[2.] A $b$-Hamiltonian $S^1$-action on the torus by the same rotation with $S^1$ being the image of $b$-moment map. (Fig. \ref{pic:t2})
    \end{itemize}
    
    \begin{figure}[h!] 
        \centering
        \begin{tikzpicture}
            \draw (-6, 0) circle (1.5);
            \draw[dashed] (-6, 0) ellipse (1.5 and 0.5);
            \draw[->] (-4, 0) -- (-2, 0);
            \draw (-1, 1.5) -- (-1, -1.5);
            
            \end{tikzpicture}
        \caption{Moment map for circle action on $S^2$}
        \label{pic:s2}
    \end{figure}
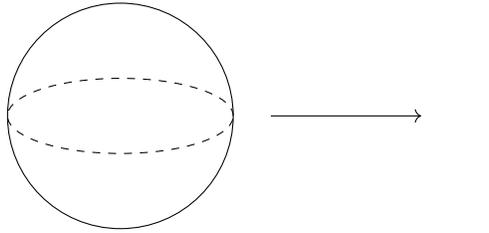
    
    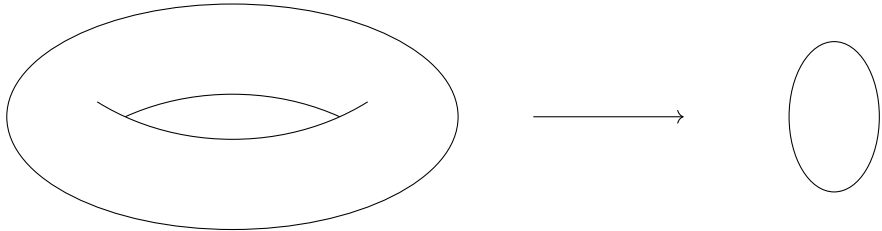
\begin{figure}[h!] 
        \centering
        \begin{tikzpicture}
            \useasboundingbox (-9,-1.5) rectangle (3,1.5);
                \draw (-6,0) ellipse (3 and 1.5);
                \begin{scope}
                    \clip (-6,-1.8) ellipse (3 and 2.5);
                    \draw (-6,2.2) ellipse (3 and 2.5);
                \end{scope}
                \begin{scope}
                    \clip (-6,2.2) ellipse (3 and 2.5);
                    \draw (-6,-2.2) ellipse (3 and 2.5);
                \end{scope}
                
            \draw[->] (-2, 0) -- (0, 0);
            
            \draw (2, 0) ellipse (0.6 and 1);
        \end{tikzpicture}
        \caption{Moment map for circle action on $T^2$}
        \label{pic:t2}
    \end{figure}
    
    As explained above, the moment map, in this case, contains a $log$-component that explodes at $h = 0$. To work with this object, one needs to introduce the notion of $b$-line and $b$-circle.

    To prescribe a smooth structure on the image of the moment map, in \cite{guillemin2015toric} the authors use $\mathbb{R}_{>0}$-valued labels ("weights") on the points at infinity. Thus, the image of the moment map forms what is called a $b$-line (or $b$-circle) constructed by gluing copies of the extended real line $\overline{\mathbb{R}} := \mathbb{R} \cup \{\pm \infty\}$ together in a zig-zag pattern where points are \emph{at infinity} are glued together.
    \begin{figure}[ht]
        \centering
        \begin{tikzpicture}[scale = 0.8]
        \foreach \copies in {0, 1, 2}{ \pgfmathtruncatemacro{\lowerlabel}{(\copies * 2) - 2}	 \pgfmathtruncatemacro{\upperlabel}{(\copies * 2) - 1}	
        \foreach \paramt in {-3, -2.75, ..., 0.01}	
	    {
	    \pgfmathsetmacro{\height}{1.5 * exp(\paramt)}	
	    \pgfmathsetmacro{\basept}{\copies * (8/3)}	
	    \pgfmathsetmacro{\xshift}{0.25*(\height * (\height - 3)) - 0.1}
	    \draw[fill = blue] (\basept + \xshift, \height - 1.5) circle(.1mm);
	    \draw[fill = blue] (\basept - \xshift, \height - 1.5) circle(.1mm);
	    \draw[red, fill = red] (\basept, -1.6) circle(.5mm) node[below right, blue] {\small{$\textrm{wt}(\lowerlabel)$}};
	
	    \draw[fill = blue] (\basept + 1.33 + \xshift, 1.5 - \height) circle(.1mm);
	    \draw[fill = blue] (\basept + 1.33 - \xshift, 1.5 - \height) circle(.1mm);
	    \draw[red, fill = red] (\basept + 1.33, 1.6) circle(.5mm) node[above right, blue] {\small{$\textrm{wt}(\upperlabel)$}};
	    }}
        \draw (7.5, 0) node[right] {\huge{$\dots$}}	;
        \draw (-0.8, 0) node[left] {\huge{$\dots$}}	;

        \pgfmathsetmacro{\rline}{11}	

        \draw (9, 0) edge node[above] {$\hat{x}$} (10.5, 0);
        \draw[->] (9, 0) -- (10.5, 0);

        \foreach \paramt in {-3, -2.75, ..., 0.01}	
	    {
	    \pgfmathsetmacro{\height}{1.5 * exp(\paramt)}	
	    \pgfmathsetmacro{\xshift}{0.25*(\height * (\height - 3)) - 0.1}
	    \draw[fill = blue] (\rline, \height - 1.5) circle(.1mm);
	    \draw[fill = blue] (\rline, 1.5 - \height) circle(.1mm);
	    }

        \draw (\rline, 1.5) node[right] {$\mathbb{R}$};
        \end{tikzpicture}
    \caption{A weighted $b$-line with $I = \mathbb{Z}$.}
    \label{fig:bline}
    \end{figure}
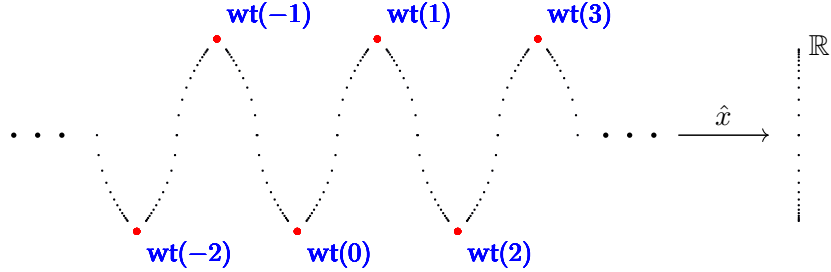

    This notion can be generalized using the language of $b$-Lie groups as introduced in \cite{BKM2}.

    \begin{definition}
        A $b$-manifold $(G,H)$, where $G$ is a Lie group and  $H \subset G$ is a closed co-dimension one subgroup\footnote{This is equivalent to $H$ being an embedded Lie subgroup.} is called a {\bf $b$-Lie group}.
    \end{definition}

   The $b$-line and the $b$-circle are themselves examples of $b$-Lie groups, with the action given by translations and rotations, respectively.

    In this article, we consider the case of general $b^m$-symplectic actions. In view of Theorem \ref{th:bmAction} for $b$-manifolds a group which acts transversally to the symplectic foliation inside $Z$ decomposes as $G=S^1\times H$ mod by a discrete group. As the only constraint is the geometric structure induced on the critical set $Z$ the same holds for $b^m$-symplectic manifolds. 
    
    By considering the restriction of $\rho$ to the $S^1$-component $\rho|_{S^1}$, we obtain a torus action on the $b^m$-symplectic manifold. 
    
    We follow \cite{GMWconvexity} and \cite{GMWgeomqbm} where the notion of modular weights of a torus action is defined and studied. 
    
    In a neighborhood $U$ of the critical set $Z$, $U=Z\times (-\epsilon, \epsilon)$,  write $\omega$ as in equation \ref{eqn:newlaurent}:

    $$
        {\omega = \sum_{j = 1}^{m}\frac{df}{f^j} \wedge \pi^*(\alpha_{j}) +  \beta }.
    $$

    Assume  that there exists a moment map $\mu\in ~^{b^m}\mathcal{C}^\infty(M)\otimes \mathfrak{t}$ with
    $$
        \langle d\mu, \xi \rangle= i_{\xi^M}\omega
    $$

    \noindent for any $\xi\in\mathfrak{t}$ and where $\xi^M$ stands for the fundamental vector field generated by $\xi$;

    \begin{definition} 
        The  modular weights $a_1, \dots, a_m\in \mathfrak{t}^*$ in each connected component of $Z$ are given by
        $$ 
            a_j(\xi)=\alpha_j(\xi^M).
        $$
    \end{definition}
    \noindent In \cite{GMWconvexity} it is shown that these are constants.

    \begin{definition} 
        The modular weights of $\rho$ are the modular weights of the induced $S^1$-action $\rho|_{S^1}$.
    \end{definition}
    \textbf{Assumption:}
    Most of the results in this article are proved under the assumption that \emph{the highest modular weight} $a_m$ is nonzero.
    
    As in \cite{GMWconvexity} it is sufficient to assume that the highest modular weight $a_m$ is nonzero at least for one connected component of $Z$. 

    By construction, $b^m$-Hamiltonian actions with  highest non-vanishing modular weight fulfill the transversality condition in \cite{BKM}.

\section{Motivating example: Atiyah-Bott space of flat connections on $b$-symplectic surfaces} \label{Sec:AB}

    One of the motivating examples for this work comes from symplectic structures on the moduli space of flat connections $\mathcal{M}^G$. There are different approaches on how to introduce symplectic structure on $\mathcal{M}^G$ (Narasimhan and Seshadri \cite{NarSes} and \cite{Ses}, Goldman \cite{Gold1} and \cite{Gold2} and Atiyah-Bott \cite{AB}). Some of these approaches can be generalized from closed oriented Riemann surfaces $\Sigma$ of genus $g$ to surfaces with boundary. However, this kind of singularities on the surface does not lead to the singularities in the symplectic structure. The question that captured our attention is if and under which condition can the Poisson structure on the moduli space of flat connections become $b$- or $b^m$-symplectic. 
    
    In this section, we follow the approach of Atiyah and Bott. As a result, we arrive to a family of examples of $b^m$-symplectic structures on the moduli space of flat connections in trivial bundles over manifolds with boundary.
    
\subsection{Flat connections over compact orientable surfaces}
    
    In order to construct an example of the singular Atiyah-Bott form, we first turn to the usual symplectic case. We remind the approach from \cite{AM} and \cite{DM} on how to explicitly write the Atiyah-Bott form on a compact oriented manifold in Darboux coordinates via holonomies corresponding to the generators of the fundamental group. The main idea underlying this approach is that the integral $\int_\Sigma \alpha \wedge \beta$ taken over the whole surface can be lifted to the universal cover and becomes equal to the integral over the fundamental polygon. The lifted integral itself can be localized to the alternating sum of the values taken in the vertexes of the polygon, that is, the sum of the holonomies corresponding to the edges of the fundamental polygon (i.e., generators of the fundamental group).

    Let $\Sigma$ be a compact orientable surface, and $G$ be a Lie group admitting $\Ad$-invariant bilinear form on its Lie algebra. Let $\Sigma \times G$ be a trivial bundle. Now we can introduce a theorem relating generators of the fundamental group of $\Sigma$ and corresponding holonomies with the Atiyah-Bott symplectic structure as Darboux coordinates.
    
    \begin{theorem} \label{th:ABDarboux}
        Suppose $\Sigma$ is a compact orientable surface of genus $g \geq 1$, and suppose $G$ is an abelian Lie group with a fixed symmetric non-degenerate bilinear form on its Lie algebra. Then the moduli space $\mathcal M(\Sigma,G)$ is diffeomorphic to $G^{2g}$.
        
        Write $a_1, b_1, \ldots, a_g, b_g$ for the $2g$ projections to the copies of $G$ composed with the inversion map $G \to G$ so that by abuse of notation $a_i \in G$ is the holonomy around the generator $a_i$ of the fundamental group (and analogously for $b_i$).Then $\theta \circ (d a_i): d \text{Hom}(\pi_1(\Sigma), G) \to \mathfrak{g}$ and $\theta \circ (d b_i) : d \text{Hom}(\pi_1(\Sigma), G) \to \mathfrak{g}$ are $\mathfrak{g}$-valued 1-forms on $\text{Hom}(\pi_1(\Sigma), G)$. Call these 1-form $d a_i$ and $d b_i$. Then the symplectic structure on the block $M_{\Sigma \times G} (\Sigma, G)$ of the moduli space is given by
        $$
            \omega _{AB} = \sum \limits_{i = 1}^g (d b_i \wedge d a_i)
        $$
    \end{theorem}
    
    \begin{example} \label{ex:torSimpl}
        Let us consider the torus $\mathbb T^2$ as an illustrating example. The corresponding fundamental polygon is a square with oriented edges (see Fig. \ref{fig:TSqSympl}) and the form $\omega_{AB} = db \wedge da$ as in theorem \ref{th:ABDarboux}.
        
        \phantom{a}
        \begin{figure}[h]
            \centering
            \begin{tikzpicture} 
                \draw (6,-1.5) -- node [vee][rotate = 90]{} (9,-1.5); 
                \draw (9,-1.5) --  node {\midarrowu}(9,1.5);
                \draw (6,1.5) -- node [vee][rotate = 90]{} (9,1.5);
                \draw (6,-1.5) -- node {\midarrowu}(6,1.5);
                
                \filldraw[black] (5.5,-0.5) circle (0pt) node[anchor=west] {a};
                \filldraw[black] (9.1,-0.5) circle (0pt) node[anchor=west] {a};
                \filldraw[black] (7,1.8) circle (0pt) node[anchor=west] {b};
                \filldraw[black] (7,-1.8) circle (0pt) node[anchor=west] {b};
            \end{tikzpicture}
            \caption{Fundamental polygon for $\mathbb{T}^2$}
            \label{fig:TSqSympl}
        \end{figure}
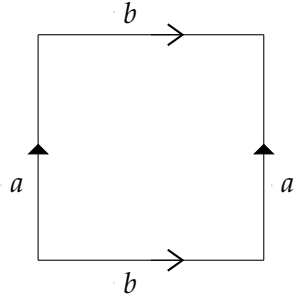
    \end{example}
    
    \subsection{The $b^m$-symplectic construction}
    
    Now we consider the singular analogue of the above example. In order to construct it, we use the desingularization Theorem \ref{th:deblog} and treat a $b^m$-symplectic structure (for even $m$) as a limit of a family of symplectic ones. 

    Take the limit of symplectic forms on $\mathbb T^2$ as in \ref{eq:b2kDesing}: 
    $$
        \lim_{\varepsilon \to 0} \omega_\varepsilon = d f_\varepsilon \wedge \left ( \sum \limits_{i = 1}^{2k} x^i \alpha_ i \right ) + \beta 
    $$
    and consider $f_\varepsilon$ being desingularizing function for a $b^2$-form $\frac{d \theta}{\sin^2(\theta/2)}$ having singularity along the hypersurface $Z = \{(\theta, \varphi)| \theta = 0\}$.
    In the case of $\mathbb{T}^2$, $k = 1$ and substitute the coordinates $(\theta, \varphi)$ into the expression, we get that $\beta = 0$ and
    $$
        \lim_{\varepsilon \to 0} \omega_\varepsilon = d f_\varepsilon \wedge  x \alpha + \beta = \frac{d \theta}{\sin^2 (\theta / 2)} \wedge d \varphi = \omega.
    $$
    
    The form on the right-hand side is a $b^2$-symplectic form on $(\mathbb T^2, S^1)$ where $S^1$ is the critical hypersurface $\alpha = 0$ corresponding to the $a-$cycle.
    
    \phantom{a}
    
    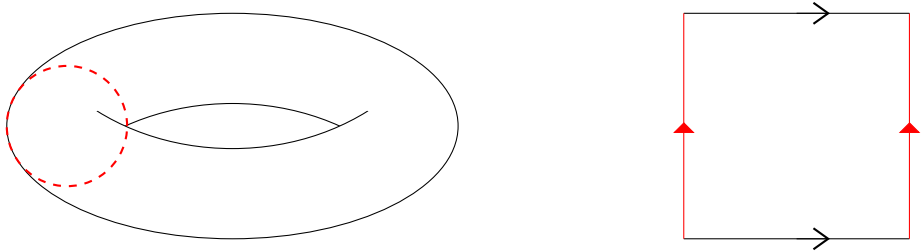
\begin{figure}[h]
        \centering
        \begin{tikzpicture}
            \useasboundingbox (-9,-1.5) rectangle (3,1.5);
            \draw (-6,0) ellipse (3 and 1.5);
            \begin{scope}
                \clip (-6,-1.8) ellipse (3 and 2.5);
                \draw (-6,2.2) ellipse (3 and 2.5);
            \end{scope}
            \begin{scope}
                \clip (-6,2.2) ellipse (3 and 2.5);
                \draw (-6,-2.2) ellipse (3 and 2.5);
            \end{scope}
            \draw[red,thick,dashed] (-8.2,0) circle (0.8cm);
        
            \draw (0,-1.5) -- node [vee][rotate = 90]{} (3,-1.5); 
            \draw[red] (3,-1.5) --  node {\midarrowu}(3,1.5);
            \draw (0,1.5) -- node [vee][rotate = 90]{} (3,1.5);
            \draw[red] (0,-1.5) -- node {\midarrowu}(0,1.5);
        \end{tikzpicture}  
        \caption{$b$-manifold $(\mathbb{T}^2, S^1)$ and the corresponding fundamental polygon}
        \label{fig:T2S1}
    \end{figure}
    
    \phantom{a}
    
    For the left-hand side of the equality, we can easily juxtapose the corresponding polygon and, therefore, the normal form of $\omega_{AB}$. Note that $\omega$ here is a form on the manifold $\Sigma$ and $\omega_{AB}$ is the form on the moduli space of flat connections $\mathcal M^G$. The family of polygons would still be squares as in example \ref{ex:torSimpl}. What changes here is the values of the holonomies (i.e., lengths of the edges), with one of them tending to infinity. In the limit, the corresponding fundamental polygon is as on Fig. \ref{fig:T2S1} where red edges correspond to the single distinguished circle $\alpha = 0$ that is the critical hypersurface for the form $\omega = \frac{d \alpha}{\alpha} \wedge d \beta$. The corresponding Atiyah-Bott form is then $\omega_{AB} = \frac{d b}{b}\wedge da$.
    
    This example is notable because the Darboux coordinates on the surface coincide with the generators of the fundamental group, which does not hold in general. 
    
    First, consider the following $b$-manifold: a torus $\mathbb T^2$ with one cycle $S^1$, chosen as a critical hypersurface. As shown in \cite{MP} such a $b$-manifold can only admit $b^{k}$-symplectic structures for even $k$. As an example, we choose $\omega = \frac{d \theta}{\sin^2 (\theta / 2)} \wedge d \varphi$. The polygon corresponding to this surface is a square (see Fig. \ref{fig:T2S1}b), where the red edges correspond to a single distinguished circle (as in Fig. \ref{fig:T2S1}a).
  
    Consider the following two  $b$-manifolds:
    \begin{itemize}
        \item $S^2$ with an equator $Z = S^1$ as an critical hypersurface
        \item $T^2$ with two copies of $S^1$ as an critical hypersurface
    \end{itemize}
    
    The corresponding polygons look as depicted below:
    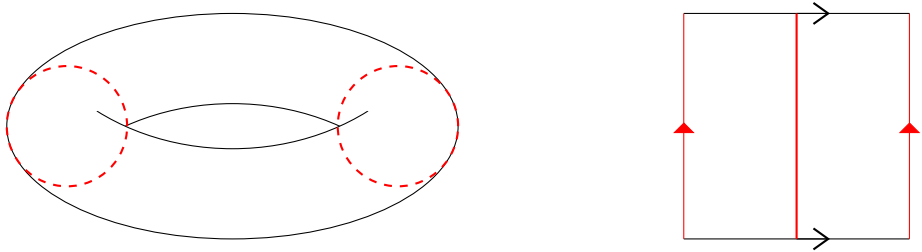
\begin{figure}[h]
        \centering
        \begin{circuitikz}
            \useasboundingbox (-9,-1.5) rectangle (3,1.5);
            \draw (-6,0) ellipse (3 and 1.5);
            \begin{scope}
                \clip (-6,-1.8) ellipse (3 and 2.5);
                \draw (-6,2.2) ellipse (3 and 2.5);
            \end{scope}
            \begin{scope}
                \clip (-6,2.2) ellipse (3 and 2.5);
                \draw (-6,-2.2) ellipse (3 and 2.5);
            \end{scope}
            \draw[red,thick,dashed] (-8.2,0) circle (0.8cm);
            \draw[red,thick,dashed] (-3.8,0) circle (0.8cm);
        
            \draw (0,-1.5) -- node [vee][rotate = 90]{} (3,-1.5); 
            \draw[red] (3,-1.5) --  node {\midarrowu}(3,1.5);
            \draw (0,1.5) -- node [vee][rotate = 90]{} (3,1.5);
            \draw[red] (0,-1.5) -- node {\midarrowu}(0,1.5);
            
            \draw[red, thick] (1.5, -1.5) to (1.5,1.5);
        \end{circuitikz}  
        \caption{A $b$-manifold $(\mathbb{T}^2, S^1 \times S^1)$ and the corresponding fundamental polygon.}
    \end{figure}
    
    \phantom{a}

    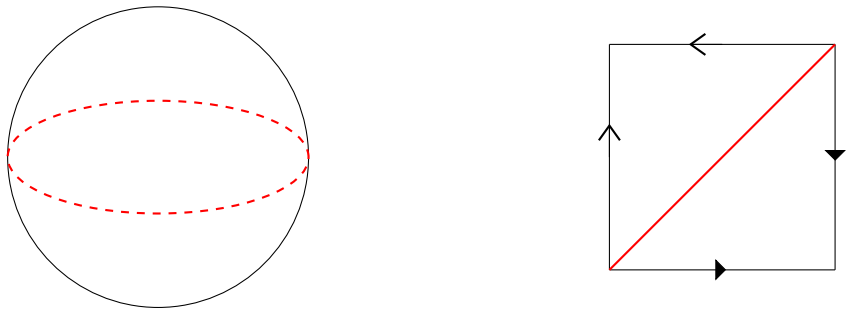
\begin{figure}[h]
        \centering
        \begin{circuitikz}
            \draw (-6, 0) circle (2cm);
        \draw[red,thick,dashed] (-6,0) ellipse (2 and 0.75);
        
            \draw (0,-1.5) to node{\midarrowr} (3,-1.5); 
            \draw (3,-1.5) -- node {\midarrowd} (3,1.5);
            \draw (0,1.5) -- node [vcc][rotate = 90]{} (3,1.5);
            \draw (0,-1.5) --  node [vcc]{}  (0,1.5);
            
            \draw[red, thick] (0, -1.5) to (3,1.5);
        \end{circuitikz}  
        \caption{A $b$-manifold $(\mathbb{S}^2, S^1)$ and the corresponding fundamental polygon.}
    \end{figure}
    
\section{ A $b^m$-symplectic slice theorem}
    
    In this section we prove the $b^m$-symplectic version of the slice theorem \ref{th:bmHsl}.
    
    Let us recall basic notions concerning slice theorems in the smooth and symplectic categories.
    
\subsection{The symplectic slice theorem}

    This section revisits different versions of slice theorems that are widely used in geometry. We will start with the general case of a slice theorem formulated by Palais in \cite{palais1, palais2} for a compact Lie group action on an abstract smooth manifold and explain the notion of a slice. Then we will consider the Guillemin-Sternberg symplectic slice theorem formulated for Hamiltonian group actions on a symplectic manifold. This theorem gives a  normal form theorem providing a semi-global description of the moment map using the slice representation. We finish this section with a short review of possible generalizations. 
    
\subsubsection{A general slice theorem for smooth actions}

    In this section, we explain the notion of slice and state the most basic slice theorem, first formulated by Palais \cite{palais1, palais2} for the group action on a general manifold and symplectic group action by Marle in \cite{marle1985modele}.
    
    Consider compact Lie group $G$ and a smooth manifold $W$ such that $G$ acts on $W$ by diffeomorphisms. For now, we do not put any restrictions on the geometric structure of $W$. It is well known that the orbits of $G$-action are submanifolds of $W$. The slice theorem allows us to describe the action in a tubular neighbourhood of an obit. We denote the orbit of point $x \in W$ by $\mathcal O_x = \{ y \in W | y = g \cdot x \text{ for some } g \in G\}$ and its stabilizer by $G_x = \{ g \in G | g \cdot x = x \}$. 
    
    A map $f_x: G \longrightarrow W$ such that $ g \longmapsto g \cdot x$ is called an orbit map. The quotient vector space $V_x = T_x W / T_x \mathcal{O}_x$ is called slice in point $x$ under the $G$-action. Slice theorem shows that the following diagram commutes
    $$
        \begin{CD}
            G/G_x @>{}>> G \times_{G_x} V_x \\
            @VV{f_x}V @VV{\Bar{f}_x}V \\
            \mathcal{O}_x @>{}>> W
        \end{CD}
    $$
    This can be stated as a theorem:
    \begin{theorem}[Slice theorem]
        There exists an equivariant diffeomorphism from an equivariant open neighborhood of the zero section in $G \times_{G_x} V_x$ to an open neighborhood of $\mathcal O_x$ in $W$, which sends the zero section $G / G_x$ onto the orbit $\mathcal{O}_x$ by the natural map $f_x$.
    \end{theorem}
    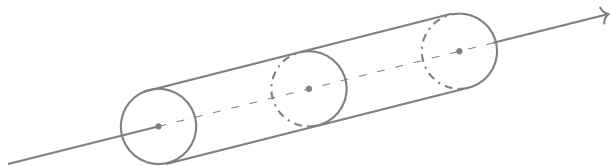
\begin{figure}[h]
        \centering
        \begin{tikzpicture}
            \draw[gray, thick] (0, 0) -- (2, 0.5);
            \draw[gray, thick] (1.8, 0.96) -- (5.9, 2);
            \draw[gray, thick] (2.1, 0) -- (6.06,1);
            \draw[gray, ultra thin, dash dot] (2, 0.5) -- (6.45, 1.6125);
            \draw[->][gray, thick] (6.45, 1.6125) -- (8, 2);
            \draw[gray, thick](2,0.5) circle (0.5);
            \draw[gray, thick, dash dot](4,1.5) arc (90:270:0.5);
            \draw[gray, thick, dash dot](6,2) arc (90:270:0.5);
            \draw[gray, thick](4,0.5) arc (-90:90:0.5);
            \draw[gray, thick](6,1) arc (-90:90:0.5);
            \filldraw [gray] (2,0.5) circle (1pt);
            \filldraw [gray] (4,1) circle (1pt);
            \filldraw [gray] (6,1.5) circle (1pt);
        \end{tikzpicture}
        \caption{A slice of an orbit}
        \label{fig:SimplSlice}
    \end{figure}    

\subsubsection{A slice theorem for Hamiltonian actions}

    In this section, we deal with the symplectic (and at the same time Hamiltonian) slice theorem formulated by Guillemin-Sternberg in \cite{GS84} and independently by Marle in \cite{marle1985modele}. Even more interestingly for us, we recall the Guillemin-Sternberg local normal form theorem \cite{GS90} that also gives a semi-global normal form for the moment map of the slice representation. 
    
    First, we start by reminding the definition of a Hamiltonian space.
    \begin{definition}
        An action of a group $G$ on a symplectic manifold $(M, \omega)$ is called \textbf{Hamiltonian} if it preserves the symplectic structure and admits an equivariant moment map $\mu: M \to \mathfrak{g}^*$ such that
        $$
            \iota(\upsilon_\xi) \omega = d\left <\phi, \xi \right>, \forall \xi \in \mathfrak{g}.
        $$
    \end{definition}
    Here $\left < , \right >$ is natural pairing identifying $\mathfrak g$ and $\mathfrak g^*$ and $\upsilon_\xi$ is generating vector field on $M$.
        
     \begin{theorem}[Guillemin-Sternberg, Marle] \label{Th:SSl}
         Let $(M, \omega, G)$ be a symplectic manifold together with Hamiltonian group action. Let $p$ be a point in $M$ such that $\mathcal O_p$ is contained in the zero level set of the moment map. Denote $G_p$ the stabilizer and $\mathcal O_p$ the orbit of $p$. There is a $G$-equivariant symplectomorphism from a neighbourhood of the zero section of the bundle $T^* G \times_{G_p} V_p$ equipped with a symplectic model to a neighbourhood of the orbit $\mathcal O_p$.
    \end{theorem}
       
    The action of the transversal element of $T^* G \times_{G_p} V_p$ is given by a cotangent lift (see \ref{sec:bmCotL}).
    
    Besides the symplectic slice theorem, we are also going to refer to the following theorem by Guillemin and Sternberg \cite{GS90} that gives not only a semilocal description of the group action in the neighbourhood of an orbit but also a normal form for the corresponding moment map.
    \begin{theorem}[Guillemin-Sternberg]
        Let $(M, \omega, \mu)$ be a Hamiltonian $G$-space. For any $p \in M$, let $H = Stab(p)$, let $K = Stab(\mu(p))$, and let $V$ be the symplectic slice at $p$. We denote by $\mathfrak{h}$ the Lie algebra of $H$ and by $\mathfrak{h}^0$ its annihilator. There exists a neighbourhood of the orbit $G \cdot p$ which is equivariantly diffeomorphic to a neighborhood of the orbit $G \cdot [e, 0, 0]$ in
        $$
            Y := G \times_H ((\mathfrak{h}^0 \cap \mathfrak{k}^*) \times V ).
        $$
        In terms of this diffeomorphism, the moment map $\mu: M \to \mathfrak{g}^*$ may be written as
        $$
            \mu([g, \gamma, v]) = Ad^{*}_{g} (\mu(p) + \gamma + \phi(v)),
        $$
        where $\phi: V \to \mathfrak{h}^*$ is the moment map for the slice representation.  
    \end{theorem}
    The symplectic form on the quotient bundle is called the MGS-symplectic form \cite{marle1985modele, GS90}, and further on will be denoted as $\omega_{MGS}$.
    
    Now we recall the statement of symplectic reduction proved by Marsden and Weinstein for free actions in \cite{MW} (check \cite{GGK} for the proof in the more general case of locally free actions that we include below). 
    \begin{theorem}[Symplectic Reduction]
        Let $(M, \omega, \mu)$ be a Hamiltonian $G$-space. Suppose that $\alpha \in (\mathfrak g^*)^G$ is a regular value for $\mu$; or, more generally, that the level set $\mu^{-1}(\alpha)$ is a manifold and $G$ acts on it (locally) freely. Then the topological space  of orbits $M_\alpha$  of the level-set $\mu^{-1}(\alpha)$ is a manifold (orbifold resp.); and there exists a unique closed two-form $\omega_\alpha$ on $M_\alpha$ such that $\pi^* \omega_\alpha = i^* \omega$, where $\pi: Z \to M_\alpha$ is the quotient map and $i : Z \to M$ is the inclusion map. The reduced form $\omega_\alpha$ is non-degenerate on $M_\alpha$ if and only if the form $\omega$ is non-degenerate on $M$ at the points of $\mu^{-1}(\alpha)$.
    \end{theorem}
    
    \begin{remark} Discrete isotropy groups appear naturally in Hamiltonian Dynamics (check, for instance, the twisted models for hyperbolic singularities of integrable systems in \cite{MZ}). This obliges us to consider actions that are not free but locally free.
    \end{remark}
    
\subsection{ The twisted $b^m$-cotangent models} \label{sec:bmCotL}
    
    We start by defining the generalization of twisted cotangent lift to the $b^m$-setting, which will be required for the proof of theorem \ref{th:bmHsl}. 
    
  The equivariant $b^m$-Moser theorem in the case of surfaces lets us visualize the $b^m$-symplectic manifold  $S^1 \times (-\epsilon,\epsilon)$ as a neighbourhood of the zero section of the cotangent bundle $T^* S^1 \cong S^1 \times \R$ with $b^m$-symplectic form given by the formula \begin{equation}\label{twistedsymplectic}
            \omega_c:= c d \theta \wedge \frac{dt}{t}.
    \end{equation} 
    This serves as one of the building blocks in our $b^m$-symplectic model for group actions. We start by introducing the (twisted) cotangent models for $b^m$-symplectic actions.

    The cotangent lift is one of the essential tools in symplectic geometry and the theory of integrable systems. It allows  lifting group actions from a manifold to  actions which are automatically Hamiltonian on the cotangent bundles. This leads to many examples of integrable systems. 
    
    We will consider a generalization of the notion of cotangent lift that lifts actions on a $b^m$-manifold to $b^m$-Hamiltonian actions on the $b^m$-cotangent bundle. This generalizes the set-up developed in \cite{KM} for $b$-symplectic manifolds.
        
    Given an action $\rho$ of a Lie group $G$ on a  $b$-manifold $(M, Z)$, one can lift it to the $b^m$-Hamiltonian action $\hat{\rho}$ of $G$ on the $b^m$-cotangent bundle $^{b^m}T^* M$.
    This can be done via the standard cotangent lift or via the twisted cotangent lift.

For the standard cotangent lift:
    The lifted action $\hat{\rho}$ is given by $\hat{\rho}_g := \rho_{g^{-1}}$ and $\pi$ is a canonical projection from $^{b^m}T^* M$ to $M$. The following diagram commutes:
    $$
        \begin{CD}
            ^{(b^m)}T^* M @>{\hat{\rho}_g}>> ^{(b^m)}T^* M \\
            @VV{\pi}V @VV{\pi}V \\
            M @>{\rho_g}>> M
        \end{CD}
    $$

    In order to introduce the twisted models we need to consider actions of products of groups $S^1\times H$ on products of manifolds $S^1\times N$. 
    The singularity of the $b^m$-form is on the fiber of the cotangent bundle of $S^1$.

We can define the twisted $b^m$-cotangent lift on the cotangent space of a torus $T^* S^1$ as follows: 

On $T^*S^1$ with standard coordinates $(\theta, t)$ we have the logarithmic Liouville one-form
$\lambda_{tw,c} = c\log |t| d\theta$ for $t\neq 0$. This induces a $b$-symplectic form on $T^\ast S^1$  with singularity on the fiber.

More generally, for $b^m$-structures,
we may consider the Liouville one form on $T^* S^1$   $$\lambda_{tw,\textbf{c}}=({c_1} \log |t| + \sum \limits_{i = 1}^{m - 1} {c_{i + 1}}\frac{t^{- i}}{i})d\theta.$$ \noindent where $\textbf{c}$ is a vector with the invariants ${c_i}$.  


The rotations on $S^1$ lift on $T^*(S^1)$ in a $b^m$-Hamiltonian fashion. For $T^*S^1$ we consider the \emph{twisted $b^m$-cotangent lift} with moment map:
    $$
    \mu_{T^*S^1} = c_1 \log |t| + \sum \limits_{i = 1}^{m - 1} c_{i + 1} \frac{t^{- i}}{i}
    $$
    and the twisted form
    $$
        \omega_{T^*S^1} = \sum \limits_1^{m} \frac{{c_i}}{t^i} d \theta \wedge d t.
    $$

    Now combine this twisted action with the standard cotangent lifted action of $H$ to the $b^m$-cotangent model. 
    The action of $S^1 \times H$ on its cotangent bundle is then $b^m$-Hamiltonian with  moment map, given by 
       $$ \mu = c_1 \log |t| + \sum \limits_{i = 1}^{m - 1} c_{i + 1} \frac{t^{- i}}{i} + \mu_0(x, y)
    $$
    with the twisted form ${\omega}$ defined as follows
    $$
        {\omega} = \sum \limits_1^{m} \frac{{c_i}}{t^i} d \theta \wedge d t + \sum \limits_2^n d x_j \wedge d y_j.
    $$

Namely, we consider  the diagonal action of products of groups $S^1\times H$ on products on manifolds $S^1\times N$ and extend this notion. Denote by $\lambda_N$ be the standard Liouville one-form on the cotangent bundle $T^* N$. We endow the product $T^* (S^1 \times N) \cong T^* S^1 \times T^* N$ with the product structure $\lambda:= (\lambda_{tw,c}, \lambda_N)$ (defined for $t\neq 0$). The form $\omega = -d\lambda$ is a $b^m$-symplectic structure  with critical hypersurface described by $t=0$.

If $(x_1,\ldots,x_{n-1})$ is a chart on $N$ and $(x_1,\ldots,x_{n-1},y_1,\ldots,y_{n-1})$ is the corresponding chart on $T^* N$, the form $\lambda$ reads as: 
\begin{equation*}\label{bmliouvilleform}
\lambda =   \left(c_1 \log |t| + \sum \limits_{i = 1}^{m - 1} {c_{i + 1}}\frac{t^{- i}}{i}\right) d\theta+\sum_{i=1}^{n-1} y_i dx_i .
\end{equation*}

Let $H$ be a Lie group acting on $N$ and consider the component-wise action of $G:=S^1\times H$ on $M:=S^1 \times N$ where $S^1$ acts on itself by rotations. We lift this action to the cotangent bundle $T^* N$ as the standard cotangent lift. This construction, where $T^* M$ is endowed with the $b^m$-symplectic form $\omega$, is called the \textbf{twisted $b^m$-cotangent lift}.

This action is Hamiltonian with moment map given by contraction of the fundamental vector fields with the one-form $\lambda$ defined above. The following proposition was proved in \cite{KM} for the $b$-case. The general case is proved using similar gymnastics.

\begin{prop}\label{prop:twmoment} The twisted $b^m$-cotangent lift on $M=S^1\times N$ is Hamiltonian with equivariant moment map $\mu$ given by
 \begin{equation}
\langle\mu(p),X \rangle := \langle \lambda_p ,X^\#|_p \rangle,
\end{equation}

where $X^\#$ is the fundamental vector field of $X$ under the action on $T^* M$.
\end{prop}

\subsection{The Slice theorem}
    Before proceeding with the proof of the $b^m$-symplectic slice theorem, we need to prove some preliminary material.
        
    We start proving the following lemma in the context of $b^m$-symplectic manifolds ( the result for $b$-symplectic manifolds can be found in Section 5 in  \cite{GMWgeomq} ). Here we provide the proof for general $b^m$-symplectic manifolds following mutatis mutandis \cite{GMWgeomq} and also consider its equivariant version, which will be needed in the proof.
        
    As we will see in the proof of the $b^m$-symplectic slice theorem, it will be sufficient to consider the case in which the critical set $Z$ is a trivial mapping torus. So each connected component $Z_i$ of the critical set $Z_i=S^1\times L$ where $L$ denotes a symplectic leaf of the cosymplectic manifold $Z$.
    \begin{lemma}\label{keylemma}
        Let $(M, Z, \omega)$ be a $b^m$-symplectic manifold endowed with an $S^1$-action with non-vanishing highest modular weight. The $b^m$-symplectic form on $Z\times (-\epsilon, \epsilon)$ can be taken to be the two-form
        \begin{equation}
	        \omega=-d\theta\wedge (\sum\limits_{i=1}^m c_i \frac{d t}{t^i}) +\gamma_L
	        \label{eq:dec}
        \end{equation}
        \noindent
        where $\gamma_L$ is the symplectic form on the symplectic leaf $L$. In case there is an action of a group $G$ by $b^m-$symplectic diffeomorphism, this decomposition can  be achieved equivariantly. So, we may assume that the form $\gamma_L$ is $G$-invariant.
    \end{lemma}

    \begin{proof}
        By using the Laurent decomposition given by equation \ref{eqn:newlaurent} of the $b^m$-form and denoting by $\theta$  the angular coordinate on $S^1$, we can assume that the $b^m$-symplectic form is written as, 
        \begin{equation}
	       -d\theta\wedge (\sum\limits_{i=1}^m c_i \frac{d t}{t^i})+\gamma_L+d\theta \wedge \beta
	        \label{eq:dec2}
        \end{equation}
        with $\gamma_L$, the symplectic form on the symplectic leaf $L$ (so, a priori, possibly depending on $t$) and $\beta$,  a one-form depending on all the coordinates.
        Denote by  $g$ the function $g=\iota_{\frac{\partial}{\partial\theta}}\beta$.

        Now replace $\beta$ by a new $\beta$ equal to $\beta-g~d\theta$ in such a way that 
        $\iota _{\frac{\partial}{\partial\theta}}\beta=0$. 

        On the other hand, as the action of $S^1$  is $b^m$-Hamiltonian. There exist a smooth function $h~\in~ C^\infty(M)$ such that
        $$
            \iota_{\frac{\partial}{\partial\theta}}\omega=d(-\sum_{i=2}^{m} \left ( c_{i} \frac{1}{i t^{i - 1}} \right )- c_1\log |t|+h),
        $$

        \noindent and hence plugging on the equation above, proves that the one form $\beta$ is indeed exact.

        \begin{equation}
	        \beta=d{h}
	        \label{eq:exact}
        \end{equation}

        Now we are going to apply  Moser's trick. For that, we take the one-parameter family of forms for $0\leq s\leq 1$:

        \begin{equation}
	        \omega_s=-d\theta \wedge (\sum\limits_{i=1}^m c_i \frac{d t}{t^i})+\gamma_L+sd({h} d\theta).
	        \label{eq:fin}
        \end{equation}

        \noindent
        For $s=1$ this form is $\omega$, and for $s=0$ the simplified form (\ref{eq:dec}).
        In order to apply Moser's trick, we need to check that $\omega_s$ is a path of $b^m$-symplectic forms. Observe that for small $\epsilon$ on an $\epsilon$-neighbourhood, the first term of {(\ref{eq:fin})} is much larger than the third. So we conclude that the form {(\ref{eq:fin})} is $b^m$-symplectic and for all $s$. By construction, their class in $b^m$-cohomology coincides:

        $$
            [\omega_s]=[\omega_0].
        $$

        \noindent
        We are now ready to apply the $b^m$-Moser theorem (Theorem \ref{theorem:bmmoser}) to conclude that $\omega_0$ and $\omega_1$ are equivariantly $b^m$-symplectomorphic.

        A final remark: As done in Section 5 in \cite{GMWgeomq}, the $2$-form, $\gamma_L$, which restricts to the symplectic form along $L$ depends in principle on $t$.

        However, the inclusion map
        $$
            i: L\rightarrow L\times(-\epsilon, \epsilon),~p\rightarrow (p,0)
        $$
        and the projection map
        $$
            \pi: L\times (-\epsilon, \epsilon)\rightarrow L, (p, e)\rightarrow p
        $$
        \noindent
        induce isomorphisms on cohomology. In other words, {$[\gamma_L]=[\pi^* i^*\gamma_L]$}. Therefore, by using the Moser path method again, we can deform the (possibly $t$-dependent) form $\gamma_L$ to $\pi^* i^* \gamma_L$. Observe that this deformation can be done in an equivariant fashion by virtue of the equivariant Moser theorem for $b^m$-symplectic manifolds (Theorem \ref{theorem:bmmoser}). Thus we can assume that {$\gamma_L=\pi^* i^* \gamma_L$} and the new $\gamma_L$ is $G$-invariant and does no longer depend on $t$ and it is just a symplectic $2$-form on $L$. 

        The $G$-action preserves the $b^m$-symplectic structure on $(M, Z)$: $g \cdot \omega = \omega$ for any $g \in G$. We need to show that $f(g \cdot (\alpha, \beta) = g \cdot f(\alpha, \beta)) = g \cdot \omega = \omega$.
        The $G$-action splits into a direct product of $S^1$- and $H$-action where the $S^1$-action preserves cosymplectic structure on $Z$.

        In other words the following diagram commutes:
        \begin{align}
        \begin{CD}
            (M, Z) @>{G}>> (M, Z) \\
            @AA{f}A @AA{f}A \\
            Z @>{G}>> Z
        \end{CD}
        \end{align}{}
        This ends the proof of the lemma.
    \end{proof}
As the critical set of a $b^m$-symplectic manifold is cosymplectic,  we can use the following result from \cite{BKM}: 
    \begin{theorem}[Braddell, Kiesenhofer, Miranda]
        Let $Z$ be a cosymplectic manifold and suppose $Z$ has a transverse $S^1$-action preserving the cosymplectic structure. Then $Z$ has a finite cover $\Tilde{Z} := S^1 \times \mathcal L$, $\mathcal L$ a leaf of the foliation, equipped with an $S^1$ action given by translation in the first coordinate for which the projection $p : S^1 \times \mathcal L \to Z$ is equivariant. To get a cosymplectic structure on the cover, one lifts the associated defining one- and two-forms.The cosymplectic structure on $Z$ is given by the quotient of a cosymplectic structure on $\Tilde{Z} = S^1 \times \mathcal L$ by the action of a finite cyclic group $\mathbb Z_k$.
    \end{theorem}
   
   Now we are ready to prove the $b^m$-symplectic version of the slice theorem.  This describes a neighbourhood of the orbit in terms of the twisted cotangent models described in Section \ref{sec:bmCotL}.     
    \begin{theorem}[A $b^m$-slice theorem] \label{th:bmHsl}
        Let $G$ be a compact group acting on a $b^m$-symplectic manifold $(M, Z, \omega)$ by  $b^m$-symplectomorphisms such that the highest modular weight is non-vanishing. Let $\mathcal{O}_z$ be an orbit of the group contained in the critical set of $M$. Then there is a neighbourhood of the zero section of an associated bundle $^{b^m}T^* G \times_{H_z \times \mathbb{Z_d}} V_z$ equipped with the $b^m$-symplectic model 
        $$
            \omega = \sum\limits_{i=1}^m c_i \frac{d t}{t^i} \wedge d \theta + \pi^*(\omega_{MGS}),
        $$
        where $t$ is a defining function for $Z$, $\pi$ is the projection $\pi: T^* S^1 \times T^* H \times_{H_z} V_z \to T^* H \times_{H_z} V_z$ and $\omega_{MGS}$ is the symplectic form on $T^* H \times_{H_z} V_z$ given by the symplectic slice theorem.
    
        The moment map for such action is given by
        $$
            \mu = c_1 \log |t| + \sum \limits_{i = 1}^{m - 1} c_{i + 1}\frac{t^{- i}}{i} + \mu_0(x, y).
        $$
        and thus can be recovered in terms of \emph{the twisted $b^m$ cotangent models} associated to the product $S^1\times H$.
    \end{theorem}
  
    \begin{remark}
        Below, we prove the $b^m$-symplectic slice theorem when the group action is $b^m$-Hamiltonian. Nevertheless, the statement holds for actions that preserve the $b^m$-symplectic structures.  This would be a particular type of quasi-Hamiltonian structure. The proof in this set-up can be found in the section \ref{sec:qHam}. 
    \end{remark}
    \begin{remark}
        The slice theorem for $b$-symplectic manifolds has been investigated by the second author of this article in \cite{BKM}. In this article, we give a new proof for $b^m$-symplectic manifolds, which differs from the one contained in \cite{BKM}. Also in this article we characterize the action in terms of the twisted  cotangent models in Section \ref{sec:bmCotL}.  This way to determine the normal form theorem is brand-new.
   \end{remark}
   
    \begin{proof}
        Without loss of generality, since the isotropy group $\Gamma$ is discrete, we can pass from the action of $(H \times S^1) / \Gamma$ to the free action of $H \times S^1$ on the finite cover of $(M, Z)$ and then we apply equivariance as in the previous theorem to conclude. 
        In view of the Lemma \ref{keylemma} the form $\omega$ splits into two parts, where $\alpha = - d \theta \wedge (\sum\limits_{i=1}^m c_i \frac{d t}{t^i})$ is a  $b^m$-symplectic form on $S^1 \times (-\epsilon, \epsilon)$ and $\beta$ is  the symplectic form on the leaf $\mathcal L$. 
        First, we consider the Hamiltonian action of $H$ separately. It is a Hamiltonian induced on the leaves on $Z$. Thus one can apply the symplectic slice theorem (\ref{Th:SSl}) and there is an $H$-equivariant neighbourhood $U_H$ of $\mathcal{O}^H_p$ which is equivariantly symplectomorphic to $T^* H \times_{H_p} V_p$ with the symplectic form $\omega_{MGS}$ on $T^* H \times_{H_p} V_p$.
            
        Consider the $b^m$-symplectic form on $T^* S^1 \times T^* H \times_{H_p} V_p$ given by
        $$
            \omega = \sum_{i = 1}^m c_i \frac{dt}{t^i} \wedge d \theta + \omega_{MGS},
        $$
        where $t$ is a defining function for $Z$.
        
        Take the quotient $b^m$-Poisson structure on $T^* (S^1 \times H) \times_{H_p \times \mathbb Z_d} V_p$ where $\mathbb Z_d$ acts on $T^* S^1$ as the twisted $b^m$-cotangent lift  of $\mathbb Z_d$ acting by translations on $S^1$ and by linear symplectomorphisms on $V_p$ and $H_p$ acts on $T^* H$ by the cotangent lift of $H_p$ acting on $H$ by translations and by linear symplectomorphisms on $V_p$.
        
        The last step of the proof is to do the projection from the universal cover of $M$ back onto the base. For that, we use that we can assume that the \emph{linear} symplectic form $\omega_{MGS}$ on $\mathcal L$ is invariant as proved in Lemma \ref{keylemma}. This ends the proof of the theorem.
    \end{proof}
    
    \begin{remark}
        We call a \textbf{normal form for the slice} the collection of $b$-manifold $(M, Z)$, associated bundle $^{b^m}T^* G \times_{(H_z \times \mathbb{Z}_d)} V_z$, $b^m$-symplectic model $\omega = \sum_1^m c_i \frac{d t}{t^i} \wedge d \theta + \pi^*(\omega_{MGS})$, and the group action $\rho$ as described in Theorem \ref{th:bmHsl}, which is linear on the slice. We denote it as a triple $(^{b^m}T^{*} G \times_{(H_z \times \mathbb{Z}_d)} V_z, \omega, \rho)$.
    \end{remark}
    
\subsection{Desingularization and slices} \label{sec:DesingSlice}
    
    We can now compare the $b^m$-symplectic slice theorem with its symplectic analogue. This will be needed to prove that desingularization commutes with reduction.
    
    Let us recall how to construct \emph{ the desingularization function} as done in \cite{GMW}.
    
    \begin{definition}
        Let $(S,Z, x)$, be a $b^{2k}$-manifold, where $S$ is a closed orientable manifold and let $\omega$ be a $b^{2k}$-symplectic form. Consider the decomposition given by the expression \ref{eq:b2kDecomp} on an $\varepsilon$-tubular neighborhood $U_\varepsilon$ of a connected component of $Z$.
        
        Let $f \in \mathcal{C}^\infty (\mathbb R)$ be an odd smooth function satisfying $f'(x) > 0$ for all $x \in [-1, 1]$ and satisfying outside that 
        \begin{equation}
            f(x) =
            \begin{cases}
                \frac{-1}{(2k-1)x^{2k -1}} - 2 \phantom{a} \mathrm{ for} \phantom{a} x < -1\\
                \frac{-1}{(2k-1)x^{2k -1}} + 2 \phantom{a} \mathrm{ for } \phantom{a} x > 1
            \end{cases}\,.
        \end{equation}
        Let $f_\varepsilon$ be defined as $\varepsilon^{-(2k - 1)} f(x/\varepsilon)$.
        
        The $f_\varepsilon$-\textbf{desingularization} $\omega_\varepsilon$ is a form that is defined on $U_\varepsilon$ by the following expression:
        \begin{equation}
        \label{eq:b2kDesing}
            \omega_\varepsilon = d f_\varepsilon \wedge \left ( \sum \limits_{i = 1}^{2k} x^i \alpha_ i \right ) + \beta.
        \end{equation}
    \end{definition}
    As $\omega_\varepsilon$ can be trivially extended to the whole manifold $S$ so that it coincides with $\omega$ outside $U_\varepsilon$, we further refer to it as a form on $S$.
    
    \begin{definition}
        Let $(S,Z, x)$, be a $b^{2k + 1}$-manifold, where $S$ is a closed orientable manifold and let $\omega$ be a $b^{2k + 1}$-symplectic form. Consider the decomposition given by the expression \ref{eq:b2kDecomp} on an $\varepsilon$-tubular neighborhood $U_\varepsilon$ of a connected component of $Z$.
        
        Let $f \in \mathcal{C}^\infty (\mathbb R)$ be a positive even smooth function satisfying $f'(x) > 0$ for $x < 0$ and $f(x) = - x^2 + 2$ if $x \in [-1, 1]$ satisfying outside $[-2, 2]$ that 
        \begin{equation}
            f(x) =
            \begin{cases}
                \frac{-1}{(2k+2)x^{2k +2}} \phantom{a} \mathrm{ for} \phantom{a} k > 0 \\
                \log(|x|) \phantom{a} \mathrm{ for } \phantom{a} k = 0
            \end{cases}\,.
        \end{equation}
        Let $f_\varepsilon$ be defined as $\varepsilon^{-(2k)} f(x/\varepsilon)$.
        
        The $f_\varepsilon$-\textbf{desingularization} $\omega_\varepsilon$ is a form that is defined on $U_\varepsilon$ by the following expression:
        \begin{equation}
            \omega_\varepsilon = d f_\varepsilon \wedge \left ( \sum \limits_{i = 0}^{2k} \pi^*(\alpha_i) x^i \right ) + \beta, 
            \label{eq:desingFold}
        \end{equation}
        where $\pi: U \to Z$ is the projection. 
    \end{definition}
    
    Notice that in both odd and even cases, the desingularization function can be chosen invariant under the group action, as the following lemma proves:
    \begin{lemma}
        \label{le:Haar}
         Given any desingularization function $f_\varepsilon$ one can always find an invariant desingularization $f_\varepsilon^G$ by averaging over the group action
         $$
            f_\varepsilon^G=\int_G f_\varepsilon d\mu,
        $$
        for $\mu$ a Haar measure.
    \end{lemma}
    
    As an immediate consequence of Lemma \ref{le:Haar} the $b^m$-symplectic slice is the symplectic slice of the symplectic slice theorem applied to the desingularized form.  In the odd case, we can apply the desingularization procedure to prove a \emph{ folded symplectic slice theorem}. 
    
    \begin{theorem}[Desingularization of normal form, even case]
        Given a $b^m$-symplectic manifold $(M, Z, \omega)$ we consider an orbit of a point $p$ under the $G$-action $\rho$. Let $(^{b^m}T^* G \times_{H_z \times \mathbb{Z}_d} V_z, \omega, \rho)$ be a normal form for the $b^m$-slice. The $b^m$-symplectic slice $V_p$ is also symplectic slice for the desingularized symplectic structure and $G$-action with respect to the invariant desingularization function $f_\varepsilon^G$. Then the triple $(^{b^m}T^* G \times_{H_z \times \mathbb{Z}_d} V_z, \omega_\varepsilon, \rho)$ is the normal form for desingularized action where
        $$
        \omega_\varepsilon = d f_\varepsilon^G \wedge \left ( \sum \limits_{i = 1}^{2k} x^i \alpha_ i \right ) + \beta
        $$ 
        defines symplectic structure in the neighbourhood of zero section of the associated bundle $T^*G \times_G V_p$.
    \end{theorem}
    
    For the folded symplectic case, we provide a weaker statement. A general folded symplectic slice theorem is not written in the literature to the authors' knowledge. Though, for the case of folded symplectic manifolds, which an invariant desingularization procedure can obtain, the $b^m$-symplectic slice theorem yields a  folded symplectic slice theorem. In this theorem, the slice $V_p$ remains the same, and the corresponding folded symplectic form is given by the formula \ref{eq:desingFold}. Moreover, due to the Lemma \ref{le:Haar}, this is true for any desingularization as it can always be done in an invariant way.
    
    In particular, this proves:
    \begin{cor}
        Let $(M, \omega_\varepsilon)$ be a folded symplectic manifold with fold $Z$ whose form can be seen as a desingularization of a $b^{2k + 1}$-symplectic form on   $(M, Z, \omega)$. Let $G$ be a Lie group that acts on $M$ and preserves the critical hypersurface $Z$. Then the folded symplectic slice $V_p$ is given by the $b^m$-symplectic slice theorem, and the corresponding folded symplectic form can be written as
        $$
             \omega_\varepsilon = d f_\varepsilon^G \wedge \left ( \sum \limits_{i = 0}^{2k} \pi^*(\alpha_i) x^i \right ) + \beta.
        $$
    \end{cor}
        
    \begin{remark}
        Semilocally, the only constraint for a folded symplectic manifold to admit a  $b^m$-symplectic manifold is that the critical set of the folded symplectic manifold should be a cosymplectic manifold.
    \end{remark}
    
\section{The $b^m$-Hamiltonian reduction} \label{sec:bmHamRed}

    In this section we are going to describe a generalisation of the Marsden-Weinstein reduction for the case of $b^m$-Hamiltonian group actions. Eventually, we adapt the approach of the book \cite{stages} for the $b^m$-Hamiltonian reduction by stages and extend it from the Hamiltonian to the $b^m$-Hamiltonian set-up. 
    
    This reduction theorem allows considering reduction by admissible Hamiltonian functions, which are not smooth. Thus, our reduction scheme supersedes other general reduction schemes such as the ones explored in \cite{CMM} or the ones of symplectic Lie algebroids in \cite{MPR}.
    
     The main theorem in this section considers the case when the highest modular case is non-vanishing. If the modular weight is non-zero, but the highest modular weight is zero, the $b^m$-Hamiltonian vector fields generated by the action vanish along $Z$. This case is not so interesting for reduction purposes and will not be considered in this article.

    When the modular weight is zero, the reduction scheme is a consequence of the main result in \cite{MPR} as we explain in subsection \ref{others} below.

    These reduction schemes are considered at points $\mu(p)$ with $p\in Z$. Away from $Z$, the standard symplectic reduction scheme is applied.
    \subsection{Two motivating examples}
    Let us start with two motivating examples extended from  \cite{guillemin2015toric}  to the $b^m$-category (see also  \cite{MP}). In the examples below we examine the image of the moment map and use it to describe the process of reduction in an intuitive manner. Both examples correspond to  circle actions on $b^m$-surfaces (completely classified as $b^m$-manifolds in \cite{MP}).
    
    \begin{example}[The $b^m$- Hamiltonian $S^2$] \label{ex:bmsphere}
        Consider the sphere $S^2$ as a  $b^m$-symplectic manifold with critical set the equator:
        
        \[(S^2, Z = \{h = 0\}, \omega=\frac{d h}{h^m}\wedge d\theta),\] with $h\in\left[-1,1\right]$ and $\theta\in\left[0,2\pi\right)$.

        Take the $S^1$-action by rotations given by the flow of $\frac{\partial}{\partial \theta}$. Let us check that this action is indeed $b^m$-Hamiltonian and let us compute the moment map.
        There are two cases to consider:
        \begin{itemize}
            \item The case $m=1$: As $\iota_{\frac{\partial}{\partial \theta}}\omega=- \frac{d h}{h}=-d( \log |h|),$ the  moment map on $M \backslash Z$ is $\mu(h,\theta)= \log |h|$.
            
            \item The case $m>1$. Then, $\iota_{\frac{\partial}{\partial \theta}}\omega=- \frac{d h}{h^m}=-d(-\frac{1}{(m-1)h^{m-1}}),$ the  moment map on $M \backslash Z$ is $\mu(h,\theta)= -\frac{1}{(m-1)h^{m-1}}$.
        \end{itemize}
         The image of $\mu$  for $m=1$ is drawn in Figure \ref{fig:S2} as two superimposed half-lines. Each point in the image has two connected components in its pre-image (one pre-image hemisphere) in contrast with the classical symplectic case. 
         
         In both cases, as we explained in section \ref{sec:zigzag} for the case $m=1$, the moment map can be understood as  a section of ${{^b}^m}C^{\infty}$ by including points ``at infinity''.
        \begin{figure}[h!]
            \begin{tikzpicture}[scale=1]
                \pgfmathsetmacro{\rlinex}{6}
                \pgfmathsetmacro{\baseptd}{8}
                \pgfmathsetmacro{\rlineybottom}{2.75}
                \pgfmathsetmacro{\rlineymid}{4.25}
                \pgfmathsetmacro{\rlineytop}{5.75}
                \pgfmathsetmacro{\vertstretch}{1.9}	
                \pgfmathsetmacro{\yshift}{4.25}	

                \def\R{1.6}
                \pgfmathsetmacro{\circlex}{1.4}
                \draw[dashed, very thick, color = magenta] (\circlex + \R, \rlineymid) arc (0:180:{\R} and {\R * .2});
                \draw[very thick, fill = pink, opacity = .6] (\circlex, \rlineymid) circle (\R);
                \draw[very thick] (\circlex, \rlineymid) circle (\R);
                \draw[very thick, color = red] (\circlex + \R, \rlineymid) arc (0:-180:{\R} and {\R * .2});
                
                \draw (\circlex + 2, \rlineymid) edge node[above] {$\mu, m = 1$} (\rlinex - 0.5, \rlineymid);
                \draw[->] (\circlex + 2, \rlineymid) -- (\rlinex - 0.5, \rlineymid);

                \draw (\rlinex, \rlineybottom) -- (\rlinex, \rlineytop) node[right] {};

                \draw[black, fill = black] (\rlinex, \rlineymid) circle(.3mm);

                \draw[dblue, fill = dblue] (\rlinex - 0.2, \rlineymid) circle (1pt);
                \draw[dblue, fill = dblue] (\rlinex + 0.2, \rlineymid) circle (1pt);

                \draw[line width = 2pt, join = round, dblue, <-] (\rlinex - 0.2, \rlineybottom - 0.15) -- (\rlinex - 0.2, \rlineymid );
                \draw[line width = 2pt, join = round, dblue, <-] (\rlinex + 0.2, \rlineybottom - 0.15) -- (\rlinex + 0.2, \rlineymid);
                \draw [very thick, ->] (\circlex, \rlineymid + 1.3*\R) ++(\R * -.5, 0) arc (180:320: {\R * .5} and {\R * .1});
                \draw [very thick] (\circlex, \rlineymid + 1.3*\R) ++(\R * -.5, 0) arc (180:0: {\R * .5} and {\R * .1});

                \coordinate (shift) at (8,0);
                \begin{scope}[shift=(shift)]
                \pgfmathsetmacro{\rlinex}{6}
                \pgfmathsetmacro{\baseptd}{8}
                \pgfmathsetmacro{\rlineybottom}{2.75}
                \pgfmathsetmacro{\rlineymid}{4.25}
                \pgfmathsetmacro{\rlineytop}{5.75}
                \pgfmathsetmacro{\vertstretch}{1.9}	
                \pgfmathsetmacro{\yshift}{4.25}	

                \def\R{1.6}
                \pgfmathsetmacro{\circlex}{1.4}
                \draw[dashed, very thick, color = magenta] (\circlex + \R, \rlineymid) arc (0:180:{\R} and {\R * .2});
                \draw[very thick, fill = pink, opacity = .6] (\circlex, \rlineymid) circle (\R);
                \draw[very thick] (\circlex, \rlineymid) circle (\R);
                \draw[very thick, color = red] (\circlex + \R, \rlineymid) arc (0:-180:{\R} and {\R * .2});
                
                \draw (\circlex + 2, \rlineymid) edge node[above] {$\mu, m = 2$} (\rlinex - 0.5, \rlineymid);
                \draw[->] (\circlex + 2, \rlineymid) -- (\rlinex - 0.5, \rlineymid);

                \draw (\rlinex, \rlineybottom) -- (\rlinex, \rlineytop) node[right] {};

                \draw[black, fill = black] (\rlinex, \rlineymid) circle(.3mm);

                \draw[dblue, fill = dblue] (\rlinex - 0.2, \rlineymid) circle (1pt);
                \draw[dblue, fill = dblue] (\rlinex + 0.2, \rlineymid) circle (1pt);

                \draw[line width = 2pt, join = round, dblue, <-] (\rlinex - 0.2, \rlineybottom - 0.15) -- (\rlinex - 0.2, \rlineymid );
                \draw[line width = 2pt, join = round, dblue, ->] (\rlinex - 0.2, \rlineymid) -- (\rlinex - 0.2, \rlineytop + 0.15);
                
                \draw[line width = 2pt, join = round, dblue, <-] (\rlinex + 0.2, \rlineybottom - 0.15) -- (\rlinex + 0.2, \rlineymid );
                \draw[line width = 2pt, join = round, dblue, ->] (\rlinex + 0.2, \rlineymid) -- (\rlinex + 0.2, \rlineytop + 0.15);
                
                \draw [very thick, ->] (\circlex, \rlineymid + 1.3*\R) ++(\R * -.5, 0) arc (180:320: {\R * .5} and {\R * .1});
                \draw [very thick] (\circlex, \rlineymid + 1.3*\R) ++(\R * -.5, 0) arc (180:0: {\R * .5} and {\R * .1});
                \end{scope}
            \end{tikzpicture}
            \caption{The moment map of the $S^1$-action by rotations  on a $b^m$-symplectic $S^2$.}
            \label{fig:S2}
        \end{figure}
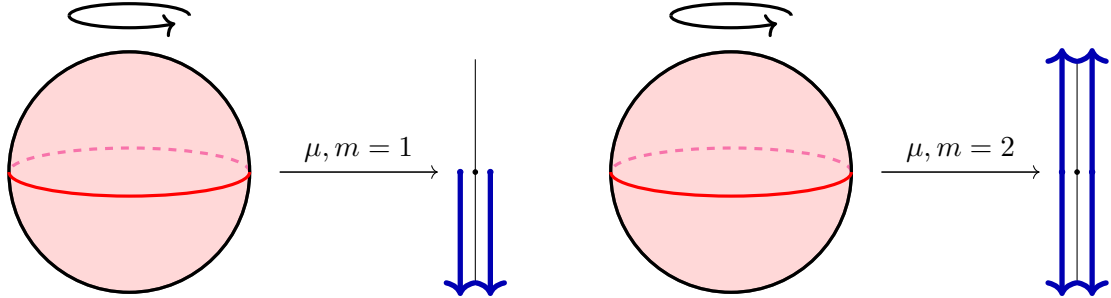
    \end{example}

    Let us yet examine another example. 
    \begin{example}\label{ex:b2torus}
Consider now as $b^2$-symplectic manifold the torus
\[
(\mathbb{T}^2, Z = \{\theta_1 \in \{0, \pi\}\}, \omega=  \frac{d\theta_1}{\sin^2\theta_1}\wedge d\theta_2)
\]
with standard coordinates: $\theta_1, \theta_2 \in \left[0, 2\pi \right)$. The critical  hypersurface $Z$ in this example is not connected. It is the union of two disjoint circles. Consider  the circle action of rotation on the $\theta_2$-coordinate with fundamental vector field $\frac{\partial}{\partial\theta_2}$. As
$$\iota_{\frac{\partial}{\partial\theta_2}}\omega = - \frac{d \theta_1}{\sin^2 \theta_1} = d\left(\frac{\cos\theta_1}{\sin\theta_1}\right).$$
Thus the associated $S^1$-action has as  $^{b^2}C^{\infty}$-Hamiltonian the function $-\frac{\cos\theta_1}{\sin\theta_1}$.

The image of this function on $M \backslash Z$ is drawn in Figure \ref{fig:torus}. Each of the two connected components of $M\setminus Z$ is diffeomorphic to an open cylinder and maps to one of these lines. Again, notice that the pre-image of a point in the image consists of two orbits.
\end{example}

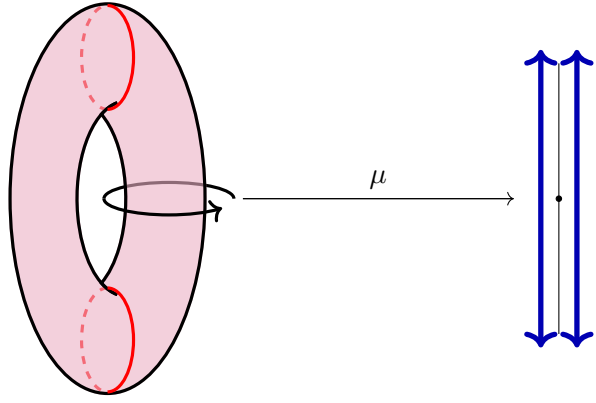
\begin{figure}[ht]
\begin{center}

\begin{tikzpicture}[scale=1.2]

\pgfmathsetmacro{\rlinex}{6}
\pgfmathsetmacro{\rlineybottom}{2.75}
\pgfmathsetmacro{\rlineymid}{4.25}
\pgfmathsetmacro{\rlineytop}{5.75}

\def\R{1.6}
\pgfmathsetmacro{\donutx}{1.5}
\pgfmathsetmacro{\sizer}{1.8}

\draw [very thick] (\donutx - 0.3*\sizer, \rlineymid) arc (180:0: {\sizer * .4} and {\sizer * .1});

\draw [red, very thick, dashed] (\donutx - 0.5, \rlineymid - .55*\sizer) arc (90:270:{.16*\sizer} and {0.32*\sizer});
\draw [red, very thick, dashed] (\donutx - 0.5, \rlineymid + .55*\sizer) arc (270:90:{.16*\sizer} and {0.32*\sizer});

\DrawFilledDonutops{\donutx - 0.5}{\rlineymid}{.6*\sizer}{1.2*\sizer}{-90}{purple!30}{very thick}{white}
\DrawDonut{\donutx - 0.5}{\rlineymid}{.6*\sizer}{1.2*\sizer}{-90}{black}{very thick}

\draw [red, very thick] (\donutx - 0.5, \rlineymid - .55*\sizer) arc (90:-90:{.16*\sizer} and {0.32*\sizer});
\draw [red, very thick] (\donutx - 0.5, \rlineymid + .55*\sizer) arc (-90:90:{.16*\sizer} and {0.32*\sizer});

\draw [very thick] (\donutx - .3*\sizer, \rlineymid) arc (180:132: {\sizer * .4} and {\sizer * .1});
\draw [very thick, ->] (\donutx - .3*\sizer, \rlineymid) arc (180:325: {\sizer * .4} and {\sizer * .1});

\draw     (\donutx + 1, \rlineymid) edge node[above] {$\mu$} (\rlinex - 0.5, \rlineymid);
\draw[->] (\donutx + 1, \rlineymid) -- (\rlinex - 0.5, \rlineymid);

\draw (\rlinex, \rlineybottom) -- (\rlinex, \rlineytop) node[right] {};

\draw[black, fill = black] (\rlinex, \rlineymid) circle(.3mm);

\draw[line width = 2pt, join = round, dblue, <->] (\rlinex - 0.2, \rlineybottom - 0.15) -- (\rlinex - 0.2, \rlineytop + 0.15);
\draw[line width = 2pt, join = round, dblue, <->] (\rlinex + 0.2, \rlineybottom - 0.15) -- (\rlinex + 0.2, \rlineytop + 0.15);
 
\end{tikzpicture}

\end{center}
 \caption{An $S^1$-action on a $b$-symplectic $\mathbb{T}^2$ and its moment map.}
 \label{fig:torus}
 \end{figure}
\begin{example} \label{ex:bmtorus}
Similarly, one can consider the torus to be a $b^m$-symplectic manifold for any integer $m$
\[
(\mathbb{T}^2, Z = \{\theta_1 \in \{0, \pi\}\}, \omega=  \frac{d\theta_1}{\sin^m\theta_1}\wedge d\theta_2).
\]
Then 
$$\iota_{\frac{\partial}{\partial\theta_2}}\omega = - \frac{d \theta_1}{\sin^m \theta_1} = d\left(\frac{|\cos\theta_1|}{\cos\theta_1} \frac{_2F_1 \left ( \frac{1}{2}, \frac{1 - m}{2}; \frac{3 - m}{2}; \sin^2(\theta_1) \right )}{(1-m) \sin^{m - 1} \theta_1}\right),$$
where $_2F_1$ is the hypergeometric function. 

Thus the associated $S^1$-action has as  $^{b^m}C^{\infty}$-Hamiltonian the function $- \frac{|\cos\theta_1|}{\cos\theta_1} \frac{_2F_1 \left ( \frac{1}{2}, \frac{1 - m}{2}; \frac{3 - m}{2}; \sin^2(\theta_1) \right )}{(1-m) \sin^{m - 1} \theta_1}$.
\end{example}

      In all the examples described above  when we fix a value of the moment map we obtain a circle (with the exception of the fixed points) where the initial $S^1$-action restricts. This circle can be quotiented out by the induced $S^1$-action to obtain a point. The singular symplectic structure in this process is also reduced to the trivial symplectic structure on the point.
    
    This would be a hands-on example of $b^m$-symplectic reduction. This reduction reduces the $b^m$-sphere/torus to a point. In higher dimensions, this reduction yields a non-trivial symplectic structure on the quotient.

\subsection{The case of non-vanishing highest modular weight}
   
    Let us now state the  $b^m$-symplectic reduction theorem of a $b^m$-Hamiltonian action with non-vanishing highest modular weight. The critical outcome of this result is that the \emph{singularity} of the $b^m$-symplectic structure is cleared away by the reduction procedure. So we could think that reduction is out of the $b^m$-symplectic category. However, we observe that reduction sits in the category of $E$-symplectic manifolds \cite{MS} where one could more generally formulate the reduction scheme. 
    
     From now on, inspired by example from Fig. \ref{pic:s2} we will introduce a notation to denote the image of the points at infinity as a boldfaced zero $\mathbf{0}$ by this, we mean the point  $\textbf{0}=(p_\infty,0)$ where the splitting of the moment map is given in view of the slice theorem (Theorem \ref{th:bmAction}), in a neighbourhood of the orbit as,   
     \begin{align}\label{tmp2}
        \mu = c_1 \log |t| + \sum \limits_{i = 1}^{m - 1} c_{i + 1} \frac{t^{- i}}{i} + \mu_0(x, y).
    \end{align}
    with symplectic form:
    \begin{align}\label{tmp3}
        \omega = \sum \limits_{i = 1}^m c_i \frac{d t}{t^i} \wedge d \theta + \pi^*(\omega_{MGS})
    \end{align}
    Thus when we consider $\mu^{-1}(\textbf{0})$ we mean the intersection of the pre-image of $\mu_0$ in the enlarged model with $t=0$ 
    
     We recall once again that the group $G$ is of the form $(S^1 \times H) / \Gamma$ that can as well be seen as $S^1 \times H$ on the universal cover of $M$.   For convenience, we will make two \underline{assumptions}:
     
     \begin{itemize}
         \item The induced action of $H$ is locally free.
         \item The action of $S^1$ on the covering model associated with the finite group $\Gamma$ is free.
         \item $0$ is a regular value for $\mu_0$ (by abuse of notation, we will then say that \textbf{0} is a regular point of $\mu$).
     \end{itemize}
   
    \begin{theorem}[The $b^m$-Marsden-Weinstein reduction] \label{th:bmHamRed}
        Given a $b^m$-Hamiltonian (locally) free action of a Lie group $G$ on a $b^m$-symplectic manifold $M^{2n}$. Assume that the highest modular weight is non-vanishing, then the pre-image of a regular point $\mu^{-1}(\textbf{0})$  is a $b^m$-presymplectic manifold that has an induced action of $G$. The space of orbits of the induced action $M//G$ is a symplectic orbifold. This reduced symplectic orbifold is symplectically isomorphic to the standard symplectic reduction of a symplectic leaf on $Z$ by a Lie subgroup of $G$.
    \end{theorem}

    \begin{proof}
    
        \textbf{Step 1: } The proof starts by setting up a structure (smooth or orbifold type) on the topological quotient. When the action is free, the quotient is a smooth manifold. If the action is locally free, the quotient has an orbifold structure (see \cite{GGK}). This orbifold structure is well-understood in the symplectic case (see for instance, \cite{GGK}).
        
        \textbf{Step 2:} Next, using the slice theorem (Theorem \ref{th:bmHsl}), we describe the induced geometrical structure on the quotient.
        If the group action is free, then a neighbourhood of the orbit is diffeomorphic to a product of the orbit with a symplectic slice. Otherwise, there is a finite group $\Gamma$ involved, and by arguing on a covering in a standard way, we can reduce to the product case.
        We apply the $b^m$-symplectic slice theorem \ref{th:bmHsl} and, more concretely, the normal form for the $b^m$-symplectic form and the moment map on a neighbourhood of $\mathcal O_x$. 
        
        Due to the $b^m$-symplectic slice theorem \ref{th:bmHsl}, the tubular neighbourhood of the orbit $\mathcal O_x$ is equipped with the following symplectic model:
        \begin{align}\label{tmp1}
            \omega = \sum \limits_{i = 1}^m c_i \frac{d t}{t^i} \wedge d \theta + \pi^*(\omega_{MGS})
        \end{align}
        and the induced moment map along the orbit has the form:
        \begin{align}\label{tmp5}
            \mu = c_1 \log |t| + \sum \limits_{i = 1}^{m - 1} c_{i + 1} \frac{t^{- i}}{i} + \mu_0(x, y).
        \end{align}
        
        Thanks to this expression:
       As the highest modular weight $c_m$ is non-zero, we can first (for simplicity and clearness) consider the reduction with respect only to the $S^1$ component (this is possible because of the normal form model in Theorem \ref{th:bmHsl}). The resulting space is a symplectic manifold which we denote as $M // S^1$ with Hamiltonian $H$-action and the corresponding induced moment map $\mu_0(x, y)$. This moment map $\mu_0$ is a standard Hamiltonian moment map. 
        
        \textbf{Final step:}
        The $H$-action on the cover can be seen as a usual Hamiltonian action on a symplectic slice so that the Marsden-Weinstein reduction can be applied directly to the second component, and the reduction $\mu^{-1}(0)/G$ is a symplectic orbifold which is symplectically equivalent to $\mathcal L//H$ (where $\mathcal L$ is any symplectic leaf on $Z$). This ends the proof of the theorem.
    \end{proof}
    \begin{remark} 
        The very general reduction scheme explained in \cite{CMM} seems to consider only smooth functions, so our approach is more general in the case of $b^m$-symplectic manifolds.
    \end{remark}
    
    \begin{remark}
        In the classical set-up of the study of Hamiltonian $G$-spaces, the Kirwan map  $\kappa:H_G^*(M)\mapsto H^*(M//G)$ defines a surjection between the equivariant cohomology of the symplectic manifold and the cohomology of the symplectic reduced space.

        Using the $b^m$-equivariant cohomology and the Mazzeo-Melrose formula, one could define the Kirwan map for $b^m$-Hamiltonian actions. Properties of this Kirwan map will be studied elsewhere.
    \end{remark}
   
\subsection{The case of vanishing modular weight}\label{others}
    The case of vanishing modular weight is easier to deal with as the action is Hamiltonian, and the $b^m$-symplectic reduction of a $b^m$-symplectic manifold is a $b^m$-symplectic manifold.
    
    For general symplectic Lie algebroids, the general Marsden-Weinstein reduction has been proved by Marrero, Padrón and Rodriguez-Olmos (see Theorem 3.11 in \cite{MPR}).
    
    By direct application of this result for the particular Lie algebroid given by the $b^m$-cotangent bundle, we obtain the following:
    
    \begin{theorem}[The $b^m$-Marsden-Weinstein reduction with zero modular weight]
        Given a $b^m$-Hamiltonian (locally) free action of a Lie group $G$ on a $b^m$-symplectic manifold $M^{2n}$ with vanishing modular weight. Then  the pre-image of a point $\mu^{-1}(0)$ is a $b^m$-presymplectic manifold that has an induced action of $G$. The space of orbits of the induced action $M//G$ is a $b^m$-symplectic manifold.
    \end{theorem}
    
    \begin{remark} 
        The reduction scheme in the zero modular weight case for $b$-symplectic manifolds is Corollary 3.10 in \cite{GZ}. This corresponds to the case when the $b^m$-Hamiltonian action is indeed Hamiltonian. So, it is a particular case of our reduction procedure.
    \end{remark}
    
\section{The desingularization procedure and reduction}
        
    In this section, we investigate the desingularization procedure of \cite{GMW} in more detail to understand how it behaves under group actions on the same manifold. This leads us to prove that  \emph{desingularization commutes with reduction}.
    
    We summarize this idea in the following diagram: 

\vspace{1mm}

\begin{center}

        \begin{tikzcd}
            & (M, Z, \omega, G)\arrow{rr}{b^m-Ham \phantom{a} red} \arrow{d}{desing} && (M // G, \omega |_{M // G}) \\
            & (M, \omega_\varepsilon, G) \arrow{urr}[near end]{Ham \phantom{a} red}\\ 
        \end{tikzcd}
    
    \end{center}
  \vspace{1mm}  
    \begin{theorem}
        The desingularization procedure commutes with the $b^m$-Hamiltonian reduction.
    \end{theorem}
    
    \begin{proof}
        Notice that by virtue of Lemma \ref{le:Haar}, we can assume that this function  is invariant by the  $G$-action $\rho$. The normal form given by the slice theorem \ref{th:bmHsl} contains precisely the associated bundle  $^{b^m}T^* G \times_{H_z \times \mathbb{Z_d}} V_z$. Theorem \ref{th:bmHsl} also provides the normal form for the moment map
        \begin{align}\label{tmp4}
            \mu = \left ( c_1 \log |t| + \sum \limits_{i = 1}^{m - 1} c_{i + 1} \frac{t^{- i}}{i} \right ) + \mu_0(x, y).
        \end{align}
        As $G$-action can be locally seen as a product $S^1 \times H$, the first components of $\mu$ (in brackets) correspond to the $S^1$-action, and the last one $\mu_0$ corresponds to $H$-action. We will refer to them as $S^1$- and $H$-components of the moment map. Notice that the $H$-component $\mu_0$ is independent of $t$ therefore, once we apply the $b^m$-Hamiltonian reduction with respect to $S^1$-action, the $S^1$-component vanishes and the $H$-component acts on the symplectic slice. Moreover, $M//S^1$ is, in fact, a symplectic manifold with Hamiltonian $H$-action and the normal form for the moment map $\mu_0$. Now we can apply the classic symplectic Marsden-Weinstein reduction. 
    \end{proof}
        The critical point to the last proof is that in the moment map normal form given by Theorem \ref{th:bmHsl}, $\mu$ splits onto two orthogonal components automatically leading us to the following corollary.
    
    \begin{cor}
        The $b^m$-Hamiltonian $G$-action admits a \emph{reduction-by-stages} procedure.
    \end{cor}
    
    \begin{proof}
        As the Marsden-Weinstein reduction commutes with the desingularization and the Marsden-Weinstein reduction for a Hamiltonian action of $G_1 \times G_2$ can be done by stages \cite{stages}, we infer that the $b^m$-Hamiltonian reduction can be done by stages. 
    \end{proof}

\section{Reduction for generalized moment maps} \label{sec:qHam}
    In this section, we would like to emphasize the universality of the proof of the reduction theorem \ref{th:bmHamRed}. Notice that this proof highly depends on the splitting property of the moment map \ref{tmp2} and uses Hamiltonian reduction as a \emph{black box} for the corresponding term of the moment map. In this section, we would like to show that different generalizations of the moment map can be used in this proof. Particularly, we focus on the group-valued moment maps \cite{AMM}, and the corresponding reduction theory \cite{BTW}. We will follow the same approach as in \cite{BTW}, showing throughout the proof that all the statements used can be generalized to a singular version (including $b^m$-type singularities). Let us first define what a group-valued moment map is. From now on, the following notation is used: $\theta^l$ and $\theta^r$ stand for the left- and right-invariant Maurer-Cartan forms, $(\cdot, \cdot)$ denotes a choice of an invariant positive definite inner product on $\mathfrak g$ and $\chi \in \Omega^3(G)$ is a canonical closed bi-invariant $3$-form $\chi = \frac{1}{12}(\theta^l, [\theta^l, \theta^l]) = \frac{1}{12}(\theta^r, [\theta^r, \theta^r])$.

    \begin{definition}
        A \textbf{quasi-Hamiltonian $G$-space} is a manifold $M$ with a $G$-action $\rho$, an invariant $2$-form $\sigma$ and an equivariant \textbf{group-valued moment map} $\Phi : M \to G$ such that:
        \begin{itemize}
            \item[(i)] $\sigma$ is equivariantly closed: $d \sigma = - \Phi^* \chi$, 
        
            \item[(ii)] the moment map condition is satisfied: $\iota(\upsilon_\xi) \sigma = \frac{1}{2}\Phi^*\left ( \theta^l + \theta^r, \xi \right )$ ,
        
            \item[(iii)] $\sigma$ is weakly non-degenerate:
                $$
                    \ker \sigma_x \cap \ker d \Phi = 0.
                $$
        \end{itemize}
    \end{definition}
 
 \begin{remark} The manifold is not necessarily symplectic. For instance $S^4$ is an $SU(2)$-quasi-Hamiltonian space and the 
 with moment map $\Phi:S^4\to
{SU}(2)\cong S^3$ the suspension of the Hopf fibration
$S^3\to S^2$ (see Appendix A in \cite{AMW}). More generally, the spin spheres $S^{2n}$ admit a quasi-Hamiltonian structure (confer \cite{HJS}). Other classical non-symplectic examples are contained in the seminal article \cite{AMM}. For instance, $D(G)$ the double of a Lie group is not symplectic if the group is compact and simply connected (as its second cohomology group vanishes).
     
 \end{remark}
  
  \begin{example}\label{ex:conjugacy}
  Conjugacy classes 
of a Lie group $G$   provide basic examples of quasi-Hamiltonian spaces. Let  $\mathcal{C}\subset G$ be a conjugacy class with the conjugation action
 of $G$. Then $\mathcal{C}\subset G$ is a quasi-Hamiltonian space with  moment map $\Phi$ given by the inclusion map into $G$. As observed in \cite{AMW} these include
all compact symmetric spaces (up to finite covers).
  \end{example}  
    \begin{remark}If the Lie group $G$ is abelian
 then these conditions imply
 that the two-form $\omega$ is automatically a symplectic form (see for instance \cite{boalch} and \cite{HJS}). 
    \end{remark}
    This definition can be generalized to consider $b^m$-type singularities along the hypersurface $Z$ of a $b^m$-manifold and, more generally, to consider $E$-manifolds introduced in \cite{ns2} and \cite{MS}.
 
    \begin{definition}
        A \textbf{singular quasi-Hamiltonian $G$-space of $b^m$-type} is a $b$-manifold $(M, Z)$ with a $G$-action $\rho$, an invariant $2$-form $\sigma\in ^{b^m}\Omega(M)$  and an equivariant moment map $\Phi : M \to G$ such that:
        \begin{itemize}
            \item[(i)] $\sigma$ is equivariantly closed: $d \sigma = - \Phi^* \chi$,
        
            \item[(ii)] the moment map condition is satisfied: $\iota(\upsilon_\xi) \sigma = \frac{1}{2}\Phi^*\left ( \theta^l + \theta^r, \xi \right )$,
        
            \item[(iii)] $\sigma$ is weakly non-degenerate: 
                $$\ker \sigma \cap \ker d \Phi = 0.$$
        \end{itemize}
    \end{definition}
    
     \begin{remark} If the Lie group $G$ is abelian
 then the singular two-form $\omega$ is automatically a $b^m$-symplectic form (or more generally an $E$-symplectic form if we replace the $b^m$-functions by $E$-functions). Then this definition of quasi-Hamiltonian generalizes the investigation of symplectic actions to the $b^m$-symplectic realm.
    \end{remark}
    
    \begin{remark}
     The form $d \sigma$ is  a smooth $3$-form in $\Omega^3(M)$ rather than a singular form in  $^{b^m} \Omega^3(M)$.
    \end{remark}
    \begin{lemma}\label{lem:7.1} The $b^m$-form 
       $\sigma$ can be decomposed as:
       $$\sigma=\alpha \wedge \frac{dt}{t^m}+\beta,$$
      \noindent where $\alpha$ is a closed smooth one-form and $\beta$ is a smooth $2$-form $\beta \in \Omega^2(M)$. 
      
    \end{lemma}
    
    \begin{proof} 
    We use Proposition 10 in \cite{guillemin2014symplectic}, the $b^m$-form 
       $\sigma$ can be decomposed as:
       $$\sigma=\alpha \wedge \frac{dt}{t^m}+\beta.$$
       
       Now from the first condition in the definition of quasi-Hamiltonian we know that $d\sigma$ is smooth (as $d \sigma = - \Phi^* \chi$, where $\chi \in \Omega^3(G)$).

        The form $d \sigma$ equals the form $d \beta$ which is a smooth form $\in \Omega^2(M)$.  This automatically implies that $d\alpha=0$ as otherwise $d\sigma$ would have a singular term.  
        
    \end{proof}
    
    The fact that $\alpha$ is closed allows to conclude that the critical set has a mapping torus structure thanks to Tischler theorem \cite{T}.  This mapping torus inherits the quasi-Hamiltonian space structure from the ambient space (in the same way the critical set of a $b^m$-symplectic manifold inherits a regular Poisson structure which is cosymplectic). Along the same lines we could prove that the critical set is   a quasi-Hamiltonian mapping torus.
    
    In order to conclude our reduction theorem we just need to apply Tischler theorem to the critical set. We recall it here for convenience:
    
    \begin{theorem}\label{thm:tic}
        Let $M^n$ be a closed  manifold endowed with a $1$-form $\beta$ which is nowhere vanishing. Then $M^n$ fibers over a circle $S^1$.
    \end{theorem}

This leads us to the following proposition:

\begin{prop} \label{prop:fibr} Let $(M,G,\sigma)$ be a closed quasi-Hamiltonian space of $b^m$-type, and let $Z$ be its critical set. Then,

\begin{enumerate}
    \item $Z$ fibers over a circle $S^1$.
    \item If the group $G$ acts transversally on the fibers of then  the group then $G$ is either of the form $S^1 \times H$ or $S^1 \times H \phantom{a} mod \phantom{a} \Gamma$, where $\Gamma = \mathbb Z_l \times \mathbb Z_k$ and $\mathbb Z_k$ is a non-trivial cyclic subgroup of $H$.
\end{enumerate}
    
\end{prop}
    
    \begin{proof}
       For the first part of the proof as a consequence of lemma  \ref{lem:7.1} the form $\alpha$ is closed. If $\sigma$ satisfies minimal non-degeneracy conditions, we may assume that $\alpha$ is nowhere vanishing. In view of Tischler's theorem \ref{thm:tic}, $M$ fibers over a circle $S^1$.
        The second part of the proposition is proved  mutatis mutandis as theorem \ref{th:bmAction} in \cite{BKM}.
    \end{proof}
    
    In view of this result it is possible to talk about transverse $S^1$-action in the quasi-Hamiltonian context as we did in former sections for $b^m$-Hamiltonian actions.
    From Lemma \ref{lem:7.1} and Proposition \ref{prop:fibr}, we get an immediate corollary:
    \begin{cor} \label{cor:bmHam}
       In a neighbourhood of the critical set $Z$, the $b^m$-form $\sigma$ can be written as $d \theta \wedge \frac{dt}{t^m} + \beta$, where $\theta$ is coordinate on $S^1$. The corresponding $S^1$-action on the covering of $M$ is $b^m$-Hamiltonian.
    \end{cor}
    
    \begin{remark} \label{rem:split}
        Notice that singular quasi-Hamiltonian moment map for the $G$-action on the covering of $M$ splits into two independent components $(\Phi_S, \Phi_H)$ corresponding to the $S^1$- and the $H$-action on the covering respectively.
    \end{remark}
    
    Now we can compute the moment map for a $b^m$-Hamiltonian $S^1$-action. The $^{b^{m}} \mathbb R$- or $^{b^{m}} S^1$-valued moment map of Example \ref{ex:bmsphere} $\mu_S = \frac{-1}{(m-1)t^{m-1}}$ for $m>1$ and $\mu_S = \log |t|$ for $m = 1$. We can extend the exponential map from $^{b^{m}} \mathbb R$ and $^{b^{m}} S^1$ to points at infinity using a standard compactification procedure. By abuse of notation we will denote this extension by $\exp$.
    
    \begin{lemma}
        The $b^m$-Hamiltonian moment map for the transverse $S^1$-action is $\Phi_S = \exp \left ( \frac{-1}{(m-1)t^{m-1}} \right )$ for $m > 1$ and $\Phi_S = |t|$ for $m = 1$. 
    \end{lemma}
    
    In particular, from now on we can talk about the $S^1$-component when referring to quasi-Hamiltonian actions which are \emph{transverse} (in the sense that the action on $Z$ acts transversally to the fibers of proposition \ref{prop:fibr} over $S^1$). Such a component, in view of \ref{cor:bmHam} is automatically $b^m$-Hamiltonian.

    Analogously to  the smooth case, one can construct a $b^m$-quasi Hamiltonian space out of a $b^m$-Hamiltonian space. 
    To do that, we first provide a relevant construction of a $2$-form $\varpi$ that is $G$-invariant and $d \varpi = - \exp^* \chi$ (see Lemma 3.3 in \cite{AMM}):
    $$
        \varpi = \frac{1}{2} \int \limits_{0}^{1} \left ( exp_s^* \theta^l, \frac{\partial}{\partial s} \exp_s^* \theta^l\right ) ds,
    $$

    where for $s \in \mathbb{R}$, the map $\exp_s : \mathfrak{g} \to G$ is defined by $\exp_s(\eta) = \exp(s \eta)$. This form is commonly used in the investigation of the correspondences between Hamiltonian and quasi-Hamiltonian spaces. We provide the $b^m$-generalizations of the statements together with the references to the original statements:
    
    \begin{lemma}[\cite{AMM}, proposition 3.4] \label{lemma:Exp}
        Let $(M, \rho, \omega, \Phi)$ be a Hamiltonian $G$-space. Then $M$ with $2$-form $\sigma = \omega + \mu^* \varpi$ and moment map $\Phi = \exp(\mu)$ satisfies all axioms of a quasi-Hamiltonian $G$-space except possibly the non-degeneracy condition. If the differential $d_\xi \exp$ is bijective for all $\xi \in \mu(M)$, then the non-degeneracy condition is satisfied as well, and $(M, \rho, \sigma, \Phi)$ is a $b^m$-quasi Hamiltonian $G$-space.
    \end{lemma}
    
    \begin{proof} 
        Equivariance of $\mu$ follows from the definition and the fact that $\varpi$ is an invariant form. The proof of the first condition that $\sigma$ is equivariantly closed is straightforward:
        $$
            d \sigma = d \omega + d \Phi^* \varpi = 0 - d \Phi^* \exp^* \chi = - \mu^* \chi.
        $$
        The moment map condition:
        $$
            \iota (\upsilon_\xi)\sigma = d (\mu, \xi) + \frac{1}{2} \mu^* \exp^* (\theta^l + \theta^r, \xi)- d (\mu, \xi) = \frac{1}{2} \Phi^* (\theta^l + \theta^r, \xi).
        $$
        The proof of the non-degeneracy condition completely follows the original proof with the only minor distinction that $\upsilon$ should be taken in $^{b^m}TM \subset TM$.
    \end{proof}
    
    \subsection{Examples}
    One of the interesting examples of quasi-Hamiltonian spaces comes from the moduli space of flat connections on surfaces. In \cite{AMM}, the authors show that the moduli space of flat connections $\mathcal M (\Sigma)$ carries a quasi-Hamiltonian structure. Consider a compact, connected surface $\Sigma$ of genus $g$ with boundary $\partial \Sigma = S^1$. The space of flat $G$-connections $\mathcal A_{flat}(\Sigma)$ on $\Sigma \times G$ is invariant under the gauge group $\mathcal G(\Sigma)$ action. The space of connections $\mathcal A(\Sigma) = \Omega^1 (\Sigma, \mathfrak{g})$  on the trivialized principal $G$-bundle $G \times \Sigma \longrightarrow \Sigma$ is a symplectic manifold, equipped with a Hamiltonian action of the gauge group $\mathcal G (\Sigma) = Map (\Sigma, G )$. The moment map $\mu : \mathcal A(\Sigma) \longrightarrow Lie(\mathcal G(\Sigma))^* = \Omega^2(\Sigma, \mathfrak g)$ is given by $\mu(A) = F_A$ , where $F_A$ is the curvature of the connection $A$. A normal subgroup $\mathcal G(\Sigma, \partial \Sigma)$ of the gauge group $\mathcal G(\Sigma)$ is defined by $\mathcal G (\Sigma, \partial \Sigma ) = {\gamma \in \mathcal G(\Sigma) | \gamma |_{\partial \Sigma} = e}$. The reduced space $X := \mathcal A (\Sigma) // \mathcal G (\Sigma, \partial \Sigma)$ is a quasi-Hamiltonian $G$-space with proper moment map $\phi$. For the $b^m$-symplectic analogue of this example in the abelian case see \ref{Sec:AB}.
    
    Other constructions are directly given in \cite{boalch} following \cite{AMM}. In particular, the following theorem characterizes when a product manifold is a quasi-Hamiltonian space with a product group. This will be specially relevant for our purposes.

\begin{theorem}[\cite{AMM}]\label{thm:newfusion}
Let $M$ be a quasi-Hamiltonian $G\times G\times H$-space,
with moment map $\Phi=(\Phi_1,\Phi_2,\Phi_3)$.
Let  the group $G\times H$ act by
the diagonal embedding described as $(g,h)\to (g,g,h)$.
Then $M$ with the two-form
\begin{equation} \label{eqn: fusion 2form}
\hat{\sigma}= \sigma - \frac{1}{2}(\Phi_1^* \theta^l, \Phi_2^* \theta^r)
\end{equation}
and moment map
$$\hat{\Phi} = (\Phi_1\cdot \Phi_2,\Phi_3):M\to G\times H$$
is a quasi-Hamiltonian $G\times H$-space.

\end{theorem}

\begin{remark} \label{rem:machineryexamples} We can easily extend this fusion procedure to singular quasi-Hamiltonian spaces.The fusion product enables to construct new quasi-Hamiltonian spaces from two given quasi-Hamiltonian spaces and combine different type of singularities to obtain $E$-quasi-Hamiltonian spaces. For the sake of simplicity we will only present fusion constructions that stay in the $b^m$-category bearing in mind that this procedure allows to obtain more general singular examples.
\end{remark}

The $b^m$-reincarnation of theorem \ref{thm:newfusion} yields:

\begin{theorem}\label{thm:bmfusion}
    Let $M_1$ be a quasi-Hamiltonian space of $b^m$-type with associated group $G\times H_1$ and $M_2$  a quasi-Hamiltonian space  with group $G\times H_2$.  Their
fusion product $M_1\circledast M_2$
is the $b^m$-singular
quasi-Hamiltonian $(G\times H_1\times H_2)$-space 
obtained from the
quasi-Hamiltonian $(G\times G \times H_1 \times H_2)$-space
$M_1\times M_2$ by
fusing the two factors of $G$.
\end{theorem}

We can use this procedure to get new examples of quasi-Hamiltonian spaces by combining classical examples in the theory of quasi-Hamiltonian spaces with singular quasi-Hamiltonian spaces.
    
    Let us consider several examples of this family of new singular fusion constructions in detail.

\begin{example}
        Take a $b^m$-Hamiltonian space from example \ref{ex:bmsphere} $(S^2, S^1, \omega_S, \mu_S)$ and a quasi-Hamiltonian space $(\mathbb T^2, S^1, \sigma_T, \Phi_T)$, where $\omega_S = \frac{d h}{h^2} \wedge d\theta$ and $\sigma_T = d\theta_1 \wedge d\theta_2  + \varpi$. The moment map for the $S^1$-action on $S^2$ is $\mu_S = -\frac{1}{h}$ and the group-valued moment map for $\mathbb T^2$ is $\Phi_T = \theta_1$. Applying lemma \ref{lemma:Exp}, we get a $b^2$-quasi Hamiltonian space $(S^2, S^1, \sigma_S, \Phi_S)$, where $\sigma_S = \omega_S + \varpi$, and the corresponding $S^1$-valued moment map $\Phi_S = \exp \mu_S$. Then the fusion product $S^2 \circledast \mathbb T^2$ is an $S^1$ $b^2$-quasi Hamiltonian space with the moment map $\Phi = \Phi_S \cdot \Phi_T = e^{-1/h} \theta_1$.
    \end{example}
    Consider now a similar example stepping out of the $b^m$-category as mentioned in Remark \ref{rem:machineryexamples}:
\begin{example}
        Take two $b^m$-Hamiltonian spaces from examples \ref{ex:bmsphere}, \ref{ex:b2torus} $(S^2, S^1, \omega_S, \mu_S)$ and $(\mathbb T^2, S^1, \omega_T, \mu_T)$, where $\omega_S = \frac{d h}{h^2} \wedge d\theta$ and $\omega_T = \frac{d\theta_1}{\sin^2\theta_1}\wedge d\theta_2$. The corresponding moment maps for $S^1$-action $\mu_S = -\frac{1}{h}$ and $\mu_T = -\frac{\cos\theta_1}{\sin\theta_1}$. Applying lemma \ref{lemma:Exp}, we get two corresponding quasi-Hamiltonian spaces $(S^2, S^1, \sigma_S, \Phi_S)$ and $(\mathbb T^2, S^1, \sigma_T, \Phi_T)$, where $\sigma_S = \omega_S + \varpi$, $\sigma_T = \omega_T + \varpi$ and the corresponding $S^1$-valued moment maps $\Phi_S = \exp \mu_S$ and $\Phi_T = \exp \mu_T$. Then the fusion product $S^2 \circledast \mathbb T^2$ is an $S^1$-$E$-quasi Hamiltonian space with the moment map $\Phi = \Phi_S \cdot \Phi_T = e^{-1/h - \cos \theta_1 / \sin \theta_1}$.
    \end{example}
    More generally, any fusion product of a $b^m$-Hamiltonian and quasi-Hamiltonian spaces will lead to a singular quasi-Hamiltonian space.
    \begin{example}
        Consider a $b^m$-Hamiltonian $G$-space $(M_1, G, \omega)$ with $\omega$ a $b^m$-symplectic form. Applying lemma \ref{lemma:Exp}, we obtain a singular quasi-Hamiltonian space $(M_1, G, \sigma)$. From theorem \ref{th:bmAction}, we can assume that in a finite covering $G=S^1\times H_1$. 
        Take $M_2$ be a quasi-Hamiltonian $S^1\times H_2$ space, then in view of theorem \ref{thm:bmfusion} the fusion product $M_1\circledast M_2$ is a $G\times H_2$ $b^m$-quasi Hamiltonian space. 
 
        Another particular case comes by
         considering as $M_1$ the former examples \ref{ex:bmsphere} and \ref{ex:bmtorus} and as $M_2$ the conjugacy classes example in Example \ref{ex:conjugacy}.
    \end{example}

We can indeed obtain quasi-Hamiltonian spaces with a $b^m$-type singularity for any prescribed group of type $S^1\times H$ as the following example shows:

\begin{example}Given a Lie group $H$, Consider a $b^m$-Hamiltonian $S^1\times H$-space $(M_1, G, \omega)$ with $\omega$ a $b^m$-symplectic form. 
 Take $M_2$ be any quasi-Hamiltonian $H$-space, then in view of theorem \ref{thm:bmfusion} the fusion product $M_1\circledast M_2$ is a $S^1\times H$-$b^m$-quasi Hamiltonian space. 
\end{example}

 We can get other constructions of more general singularities via  fusion product examples using Remark  \ref{rem:machineryexamples}.
   
   \subsection{Reduction for singular quasi-Hamiltonian spaces}
   
    In the work \cite{BTW}, the authors prove a reduction theorem for quasi-Hamiltonian spaces using two auxiliary statements from \cite{AMM}. 
    
    Now let us provide the singular quasi-Hamiltonian reduction theorem as stated in \cite{AMM} and \cite{boalch}:

    \begin{theorem}[singular quasi-Hamiltonian reduction] \label{th:qHred}
        Let $M$ be a singular quasi-Hamiltonian $G_1 \times G_2$-space with non-vanishing  highest modular weight (i.e. one of the components of the product includes transverse $S^1$-action), a  singularity of $b^m$-type and the moment map $(\Phi_1,\Phi_2): M \to G_1 \times G_2$. Let $f \in G_1$ be a regular value of the moment map $\Phi_1: M \to G_1$ and $Z_f \subset G_1$ be its centralizer. Then the pull-back of the $2$-form $\sigma \to \Phi_1^{-1}(f)$ descends to the reduced space 
        $$
        M_f = \Phi_1^{-1} (f) / Z_f
        $$ 
        and makes it into quasi-Hamiltonian $G_2$-space. If $(M, \sigma, \Phi, G)$ satisfies all conditions from the definition of a quasi-Hamiltonian space except weakly non-degeneracy condition, so does the resulting reduced space. 
        In particular, if $G_2$ is abelian then $M_f$ is symplectic.
    \end{theorem}
    
    For the non-singular version see Theorem 5.1 in \cite{AMM}. We follow the same frame of the proof.
    \begin{proof}
        First, we need to show that $i^* \omega$ is $Z_f$-basic, where $i$ is the embedding $\mu_1^{-1} (f) \hookrightarrow M$. Take $\pi$ to be the projection from $\mu_1$ onto $M_f$. Taking $\xi \in \mathfrak{z}_f$, the Lie algebra of $Z_f$, we get 
        $$
            \iota(\upsilon_\xi) i^* \omega = i^* \iota(\upsilon_\xi) \omega = i^* \mu_1^* (\theta^l + \theta^r, \xi) = 0.
        $$
        \underline{Condition (i):} Take $\chi_1$ and $\chi_2$ be the canonical $3$-forms for $G_1$ and $G_2$ respectively. Then
        $$
            \pi^* d \omega_f = d i* \omega = i^* d \omega = - i^* (\Phi_1 \chi_1 + \Phi_2 \chi_2) = - i^* \Phi_2 \chi_2 = - \pi^* (\Phi_2)^*_f \chi_2.
        $$

        \underline{Condition (ii):} Let $\sigma_f \in ^{b^m} \Omega^2 (M_f)^{G_2}$ be the unique $2$-form such that $\pi^* \sigma_f = i^* \sigma$. The restriction $i^* \Phi_2$ is $Z_f \times G_2$-invariant and descends to an equivariant map $(\Phi_2)_f \in ^{b^m} \mathcal C^\infty (M_f, G_2)^{G_2}$.
    \end{proof}

    Now we can provide the last statement needed to finalize our reduction theorem. 
    
    \begin{definition}
        Let $(M, \sigma, \Phi)$ be a quasi-Hamiltonian $G$-space. Given a subspace $W \subset T_x M$, let $W^\sigma \subset T_x M$ denote the subspace of $\sigma$ orthogonal vectors. The symplectic slice at $p \in M$ is the vector space
        $$
            V = (T_p \mathcal O)^\sigma / \left ( T_p \mathcal O \cap (T_p \mathcal O)^\sigma \right),
        $$
        where $\mathcal O = G \cdot p$ is the $G$ orbit of $p$. 
    \end{definition}
    Notice that, even if $M$ is not a symplectic space, by the axioms for quasi-Hamiltonian $G$-spaces, the kernel of $\sigma_p$ is contained entirely in $T_p \mathcal O$. Hence $V$ is a symplectic vector space.
    
    \begin{theorem}[quasi-Hamiltonian Slice Theorem, Bott-Tolman-Weitsman] \label{th:BTWSl} 
        Let $(M, \sigma, \Phi)$ be a quasi-Hamiltonian $G$-space. For any $p \in M$, let $H = Stab(p)$, $K = Stab(\Phi(p))$, and $V$ be the symplectic slice at $p$. There exists a neighbourhood of the orbit $G \cdot p$ which is equivariantly diffeomorphic to a neighborhood of the orbit $G \cdot [e, 0, 0]$ in
        $$
            Y := G \times_H ((\mathfrak h^\perp \cap \mathfrak k) \times V).
        $$
        In terms of the diffeomorphism, the moment map can be written as 
        $$
            \Phi([g, \gamma, v]) = Ad_g(\Phi(p) \exp(\gamma + \varphi (v)),
        $$
        where $\varphi : V \to \mathfrak h^* \simeq \mathfrak h$ is the moment map for slice representation.
    \end{theorem}
    
    Finally, we can use the quasi-Hamiltonian reduction to complete the statement and the proof of the $b^m$-quasi-Hamiltonian reduction theorem. In definition \ref{def:bmHamMM}, the moment map $\mu$ is an element in  $ ^{b^m} \mathcal{C}^\infty (M) \otimes \mathfrak{g}^*$. In the quasi-Hamiltonian realm, one can consider $\Phi \in ^{b^m} \mathcal{C}^\infty (M) \otimes G$ to be a suitable moment map. 
    
    This allows us to investigate the case of a group action on a $b$-manifold when the $H$-component is not $b^m$-Hamiltonian. For the quasi-Hamiltonian reduction we will abuse the notation of Marsden-Weinstein reduction even though the reduced space $M // G$ will not be necessarily symplectic (see Theorem \ref{th:bmHamRed} and Corollary \ref{cor:Abel}). It can be seen as a singular quasi-Hamiltonian with the moment map taking values in $S^1 \times H$. Then Theorem \ref{th:qHred} allows us to do reduction by stages even without having a splitting of the moment map at our disposal as in Theorem \ref{th:bmHamRed}. 
    Notice that the transverse $S^1$-action is  $b^m$-Hamiltonian in view of Corollary \ref{cor:bmHam}).
    Theorem \ref{th:qHred} declares $M//S^1$ as a quasi-Hamiltonian space with an induced $H$-action. Bearing this in mind, we consider solely the $S^1$-action on $M$ and perform the singular Marsden-Weinstein reduction depicted in Theorem \ref{th:bmHamRed} with respect to this circle action.
    
    As a result, we notice two essential properties of the reduced space: it is a quasi-Hamiltonian space endowed with an $H$-action, and the singularity has been eliminated from the singular form. This means that the reduction is a honest quasi-Hamiltonian space. This allows us to use the Bott-Tolman-Weitsman quasi-Hamiltonian reduction Theorem \ref{th:BTWSl} directly.
    
     We denote by $\mathbf{f_0}$ a regular value of the moment map. In view of Corollary \ref{cor:bmHam} as in Remark \ref{rem:split}, the moment map splits.
     
    This leads us to the final statement of our main theorem under the following \underline{assumptions}:
    \begin{itemize}
        \item The induced action of $H$ is locally free.
        \item The action of $S^1$ on the covering model associated to the finite group $\Gamma$ is free.
        \item The first component of $\mathbf{f_0}$ is $\exp (0)$ (a regular value  for the induced $S^1$-action).
     \end{itemize}
     
    \begin{theorem}[ singular quasi-Hamiltonian reduction of $b^m$-type] \label{th:bmHamRedquasi} 
        Given a singular quasi-Hamiltonian space with a $b^m$-type singularity $(M, \sigma, Z)$ and a transverse $G$-action with group-valued moment map $\Phi$. If the highest modular weight for the $S^1$-component of the $G$-action is non-zero, the pre-image of a regular point $\Phi^{-1}(\mathbf{f_0})$ admits an induced action of $G$. The space of orbits of the induced action $M//G$ is quasi-Hamiltonian.
    \end{theorem}

When the group $G$ is abelian (confer \cite{HJS} and \cite{boalch} for details in the standard quasi-Hamiltonian case), the theorem above yields a honest symplectic orbifold by reduction:
    \begin{cor} \label{cor:Abel}
        If the group $G$ is abelian, the reduced quasi-Hamiltonian space $M // G$ is a symplectic orbifold.
    \end{cor}

\end{document}